\documentclass{article}
\usepackage{geometry}               
\usepackage[english]{babel}
\usepackage[utf8]{inputenc}
\usepackage{amssymb,amsmath,amsthm,amsfonts}
\usepackage{graphicx}
\usepackage[usenames,dvipsnames]{xcolor}
\usepackage{cancel}
\usepackage[colorlinks]{hyperref}
\usepackage{enumitem}
\usepackage{comment}

\usepackage[sc]{mathpazo}

\newcommand{\N}{\mathbb{N}}

\newcommand{\R}{\mathbb{R}}

\newcommand{\eps}{\varepsilon}

\newcommand{\loc}{\text{loc}}

\newcommand{\Mcal}{{\mathcal{M}}}

\newcommand{\dive}{\diverg  }

\def\XXint#1#2#3{{\setbox0=\hbox{$#1{#2#3}{\int}$ }
\vcenter{\hbox{$#2#3$ }}\kern-.6\wd0}}

\DeclareMathOperator{\diverg}{div}

\newtheorem{proposition}{Proposition}[section]
\newtheorem{theorem}[proposition]{Theorem}
\newtheorem{corollary}[proposition]{Corollary}
\newtheorem{lemma}[proposition]{Lemma}

\theoremstyle{definition}
\newtheorem{definition}[proposition]{Definition}
\newtheorem{remark}[proposition]{Remark}

\numberwithin{equation}{section}
\newcommand{\beq}{\begin{equation}}
\newcommand{\eeq}{\end{equation}}
\newcommand{\ben}{\begin{enumerate}}
\newcommand{\een}{\end{enumerate}}
\newcommand{\bit}{\begin{itemize}}
\newcommand{\eit}{\end{itemize}}

\title{Normalized solutions for the nonlinear\\ Schr\"{o}dinger equation with potential:\\ the purely Sobolev critical case}

\author{Gianmaria Verzini and Junwei Yu}

\date{\today}

\begin{document}
\maketitle

\begin{abstract}
We study the existence and multiplicity of positive solutions in $H^1(\R^N)$, $N\ge3$, with prescribed $L^2$-norm,  
for the (stationary) nonlinear Schrödinger equation with Sobolev critical power nonlinearity. It is well known that, in the free case, the associated energy functional has a mountain pass geometry on the $L^2$-sphere. This boils down, in higher dimensions, to the existence of a mountain pass solution which is (a  suitable scaling of) the Aubin-Talenti function. In this paper, we consider the same problem, in presence of a weakly attractive, possibly irregular, potential, wondering (i) whether a local minimum solution appears, thus providing an orbitally stable family of solitons, and (ii) if the existence of a mountain-pass solution persists. 
We provide positive answers, depending on suitable assumptions on the potential and on the mass value.

Moreover, by the Hopf-Cole transform, we give some applications of our results to the existence of 
multiple solutions to ergodic Mean Field Games systems with potential and quadratic Hamiltonian.
\end{abstract}
\noindent
{\footnotesize \textbf{AMS-Subject Classification}}.
{\footnotesize 35J20, 35B33, 35Q55, 35Q89, 35J61.}\\
{\footnotesize \textbf{Keywords}}.
{\footnotesize Energy critical Schr\"odinger equations, constrained critical points, solitary waves, ergodic Mean Field Games systems, weakly attractive potential.}

\section{Introduction}

We consider the problem
\begin{equation}\label{eq:main}
\begin{cases}
-\Delta U +V(x)U = \lambda U + |U|^{2^{*}-2}U  & \text{in }\mathbb{R}^{N},\smallskip\\
\displaystyle \int_{\mathbb{R}^{N}} U^2\,dx = \rho^{2},
\end{cases}
\end{equation}
in the unknown $(U,\lambda)\in H^1(\R^N)\times\R$, where $\rho>0$, $N\geq3$, $2^{*}:=\frac{2N}{N-2}$ is 
the Sobolev critical exponent, and $V$ is a weakly attractive potential, in the sense specified below.

Solutions of \eqref{eq:main} correspond to critical points of the energy functional 
$E: H^{1}(\mathbb{R}^{N})\rightarrow \mathbb{R}$, defined by
\begin{equation}\label{func:main}
E(U)=\frac{1}{2}\int_{\mathbb{R}^{N}}|\nabla U|^{2}dx + \frac{1}{2}\int_{\mathbb{R}^{N}}V(x) U^{2}dx-\frac{1}{2^{*}}\int_{\mathbb{R}^{N}}|U|^{2^{*}}dx
\end{equation}
constrained to the $L^{2}-$sphere
\begin{equation}\label{manifold:rho}
\mathcal{M}_{\rho}=\left\{U\in H^{1}(\mathbb{R}^{N}): \int_{\mathbb{R}^{N}} U^{2}=\rho^{2}\right\},
\end{equation}
the number $\lambda$ playing the role of a Lagrange multiplier.

Solutions of this kind, either of \eqref{eq:main} or of more general nonlinear Schr\"odinger 
equations, are usually called \emph{normalized solutions} because of the $L^2$-mass constraint. They 
are of interest for their relevance in the dynamics of the associated evolutive equation. In particular, they 
allow to build standing wave solutions to the evolutive NLS, and the 
orbital stability of such solitons is related with the variational characterization of the solutions: 
as a rule of thumb, families of (either global or local) minimizer of $E$ on $\Mcal_\rho$ should 
correspond to stable sets, while saddle points give rise to unstable solitons. 

For more general equations, the variational properties of the energy functional restricted on 
$\Mcal_\rho$ are strongly influenced by the growth of the nonlinear term in the equation. Let us 
consider for a moment the free, pure power energy
\[
E_p(U)=\frac{1}{2}\int_{\mathbb{R}^{N}}|\nabla U|^{2}dx -\frac{1}{p}\int_{\mathbb{R}^{N}}|U|^{p}dx,
\]
with $2<p\le 2^*$. By the Gagliardo-Nirenberg inequality one can see that, in the mass-subcritical case
$2<p<2+4/N$, $E_p$ is bounded below and coercive on $\Mcal_\rho$, 
and it admits global minimizers for every $\rho>0$; on the other hand, in the mass-supercritical case 
$2+4/N<p\le 2^*$, then $E$ has a mountain pass geometry on $\Mcal_\rho$ and it actually admits a mountain 
pass solution, again for every $\rho>0$ (in the Sobolev critical case this latter fact is true only in dimension $N\ge5$, since the Aubin-Talenti functions are not in $L^2(\R^N)$ for $N=3,4$). It is worth mentioning that, as a matter of fact, the typical 
mountain pass geometry allows also to construct bounded local minimizing sequences for the corresponding 
functional. In the case of $E_p$, anyway, such sequences weakly converge to $0$, which is not in 
$\Mcal_\rho$, so the presence of a mountain pass solution is not associated to that of a local minimizer. 
Under this perspective, one can consider perturbations or more general versions of $E_p$, dealing with more 
general nonlinearities, or including some non-constant potential term, trying to generalize the above 
results. From this point of view, the mass-subcritical case has a quite long history, as it was treated 
by Lions as one of the first examples of application of his concentration-compactness technique 
\cite{LionCC2}, and generalized in different directions, until the recent contribution by Ikoma and 
Miyamoto \cite{MR4064338}. Concerning the mass supercritical, Sobolev subcritical case, mountain pass 
solutions were first considered by Jeanjean \cite{MR1430506} and, more recently, also 
orbitally stable local minimizers were detected, in bounded domains \cite{ntvAnPDE}. Subsequently, a 
constantly increasing number of papers deals with this kind of questions, and the topic of normalized 
solutions has become a hot one. 

In the mass supercritical case, and especially in the Sobolev critical one, two questions arise quite 
naturally, for a perturbation of $E_p$: on the one hand, if the perturbation gives rise to a 
local minimum solution, thus providing an orbitally stable solitary wave; on the other hand, if the existence of a 
mountain-pass solution persists (or is created, in lower dimension). After the recent papers by Soave \cite{MR4107073,MR4096725}, 
questions of this kind have become very popular for NLS equations with combined nonlinearities, i.e. with 
energy 
\begin{equation}\label{eq:E_comb_nonl}
E_{p,q}(U)=\frac{1}{2}\int_{\mathbb{R}^{N}}|\nabla U|^{2}dx -\frac{b}{q}\int_{\mathbb{R}^{N}}|U|^{q}dx 
-\frac{1}{p}\int_{\mathbb{R}^{N}}|U|^{p}dx,
\end{equation}
with $b\in\R$ and $2<q<p\le 2^*$. In particular, in case $p=2^*$, $q$ is mass-subcritical and $b>0$, 
Soave obtains in \cite{MR4096725} the existence of a local minimizer under a smallness condition about a 
combination of $\rho$ and $b$, leaving as an open problem the existence of a mountain-pass solution. 
This latter problem has been later solved by Jeanjean and Le \cite{MR4476243} in dimension $N\ge4$ and by 
Wei and Wu \cite{MR4433054} for $N=3$.  Notice that, in the Sobolev critical case $p=2^*$, the lack of 
compactness due to working on $\R^N$ is coupled with that provided by the Sobolev critical term. 
Moreover, the existence of the mountain-pass solution is more elusive, as it requires a description of 
the lack of compactness of Palais-Smale sequences at an higher energy level.

In the present paper we address the two questions mentioned above, in the case of the functional 
$E$ introduced in \eqref{func:main}, with the aim to parallel the results for $E_{2^*,q}$ contained 
in \cite{MR4096725,MR4476243,MR4433054}. The presence of the quadratic potential term 
$\int_{\mathbb{R}^{N}}V(x) U^{2}dx$, instead of the superquadratic perturbation 
$-b\int_{\mathbb{R}^{N}}|U|^{q}dx$, reflects on several different difficulties that we have to face.
We will come to these later (see Remark \ref{rmk:comparisonwithcombined}), after the statement of our assumptions and results.

In this paper, we deal with a potential $V \in L^{\frac{N}{2}}(\R^{N})+L^{\infty}(\R^{N})$. 
More precisely, throughout the whole paper we will always assume that
\begin{equation}\label{def_pot}
\begin{aligned}
  &V(x)=V_{1}(x)+V_{2}(x), \quad &&W(x):=V(x)|x|=W_{1}(x)+W_{2}(x)\\
  \text{with }\quad&V_{1}\in L^{\frac{N}{2}}(\R^{N}), &&W_{1}\in L^{N}(\R^{N})\\
  \text{and }\quad&V_{2},W_2 \in L^{\infty}(\R^{N}), &&|V_{2}(x)|+|W_{2}(x)|\rightarrow0,  \mbox{ as }  |x|\rightarrow\infty.
\end{aligned}
\end{equation}
Moreover, we always assume that $V$ is weakly attractive, in the sense that 
 \begin{equation}\label{Assump_Neg}
 \tilde{\lambda}_{1}:=\inf\left\{\int_{\mathbb{R}^{N}}(|\nabla \varphi|^{2}+V(x) \varphi^{2})dx: \varphi\in \Mcal_1
 \right\}< 0
  \end{equation}
(with $\Mcal_1$ defined as in \eqref{manifold:rho}). In particular, the negative part $V^-$ of the potential has to be large enough, in a suitable sense.
\begin{remark}\label{rem:LL}
It is well known that, by assumptions \eqref{def_pot}, \eqref{Assump_Neg}, there exist $\tilde{\lambda}_{1}<0$ and 
$\psi_{1}\in \Mcal_1$ with $\psi_{1}\geq 0$ in $\R^{N}$, such that
  \begin{equation}\label{eq:nomain2}
    -\Delta \psi_{1}+V\psi_{1}= \tilde{\lambda}_{1}\psi_{1}\qquad\text{ in }H^1(\R^N)
      \end{equation}
(see e.g. \cite[Section 11.5]{LiebLoss}).
\end{remark}

To state our results we need the following definition.
\begin{definition}\label{def:GS}
We say that $U_{0}$ is a  normalized ground state of \eqref{eq:main} if $U_{0}$ is a critical point of $E$ constrained to $\mathcal{M}_\rho$, i.e. it solves \eqref{eq:main}, and
\begin{equation*}
E(U_{0})= \inf\{E(U): U \in \mathcal{M}_\rho, \ \nabla_{\mathcal M_\rho}E(U)=0\}.
\end{equation*}
\end{definition}

First, we state our main results concerning local minimizers and ground states.
\begin{theorem}\label{Thm_neg}
Let $N\ge3$ and assumptions \eqref{def_pot}, \eqref{Assump_Neg} hold true. 
There exists a positive constant $C_{0}$ such that, if
  \begin{equation}\label{AS_Positive}
    \max\left\{\|V^{-}_{1}\|_{\frac{N}{2}},\rho^{2} \|V_{2}^{-}\|_{\infty}\right\}<C_{0},
  \end{equation}
then \eqref{eq:main} has a solution, which corresponds to a local minimizer of $E$ on $\Mcal_{\rho}$ 
with negative energy, negative multiplier, and that can be chosen to be strictly positive almost 
everywhere in $\R^N$. 
  
If in addition    
 \begin{equation}\label{AS_N4GS1}
    \max\left\{\|\left [(N-4)V_{1}\right]^{+}\|_{\frac{N}{2}}, \|W_{1}\|_{N}\right\}<C_{1},
  \end{equation} 
    \begin{equation}\label{AS_N4GS2}
    \max\left \{\rho^{2}\| \left [(N-4)V_{2}\right]^{+} \|_{\infty},\rho \|W_{2}\|_{\infty}\right\}<C_{2},
  \end{equation}  
for appropriate positive constants $C_{1},C_{2}$, then the above local minimizer is a ground state.
\end{theorem}
\begin{corollary}\label{coro:coro1}
Let $N\ge3$ and assumptions \eqref{def_pot}, \eqref{Assump_Neg} hold true. 
There exists a positive constant $\rho^*$ such that, if $0<\rho<\rho^*$ then \eqref{eq:main} has a 
solution, which corresponds to a local minimizer of $E$ on $\Mcal_{\rho}$ 
with negative energy, negative multiplier, and that can be chosen to be strictly positive almost 
everywhere in $\R^N$. 

Moreover, for a possibly smaller threshold $\rho^*>0$, such local minimizer is a ground state.
\end{corollary}
\begin{corollary}\label{coro:coro2}
The results in Corollary \ref{coro:coro1} hold true also replacing assumption \eqref{def_pot} with 
the following one:
\[
V\in L^{r_1}(\R^{N}), \quad W\in L^{r_2}(\R^{N}),\qquad \text{with $\frac{N}{2}<r_1<+\infty$, $N<r_2<+\infty$}.
\]
\end{corollary}

Concerning the mountain pass solutions, our main result is the following.
\begin{theorem}\label{thm:mountain_pass}
Let $3\leq N \leq 5$ and suppose that  all the assumptions of either Theorem \ref{Thm_neg}
or one of Corollaries \ref{coro:coro1}, \ref{coro:coro2} are satisfied (with possibly smaller 
constants $C_1,C_2,\rho^*>0$), so 
that \eqref{eq:main} admits a ground state at negative energy, which is also a local minimizer. Assume 
moreover that $V$ is not 
unbounded in every open set, i.e. that
\begin{equation}\label{eq:final_ass}
V \in L^{\infty}(B_{2R}),
\end{equation}  
for some ball $B_{2R}$. Then \eqref{eq:main} has a mountain pass solution at a positive energy level, with negative Lagrange multiplier, which can be chosen to be strictly positive almost everywhere in 
$\R^N$, too.
\end{theorem}
\begin{remark}\label{rmk:no_neg_sol_intro}
Notice that assumption \eqref{Assump_Neg} requires $V^-$ to be big enough. For instance, it is well 
known that a sufficient condition for \eqref{Assump_Neg} to be satisfied is that, for some positive 
$r$, $\eta$,  
\begin{equation}\label{eq:suff_cond_neg}
\inf_{B_r(x_0)}V^-\ge \eta > \frac{N(N-2)}{r^2}
\end{equation}
(see e.g. \cite[Theorem B (i)]{MR4064338}). 
On the other hand, assumptions \eqref{AS_Positive}, \eqref{AS_N4GS1} and \eqref{AS_N4GS2} require 
$V_1$, $W_1$, $\rho^2 V_2$ and $\rho W_2$ to be small, in a suitable sense. Since 
the decomposition in \eqref{def_pot} is not unique, this last property can be obtained for a 
potential $V$ satisfying \eqref{Assump_Neg}, by possibly changing the decomposition, 
taking $V_2$, $W_2$ large and $V_1$, $W_1$ small, and then $\rho$ small enough. Actually, this is 
the content of Corollary \ref{coro:coro1}. For more details on such decomposition, see Lemma \ref{lem:VinLr} below.

In particular, if $V,W\in L^{\infty}(\R^{N})$, both vanishing at infinity, and $V$ satisfies 
\eqref{eq:suff_cond_neg}, the theorems apply, with $V_1\equiv W_1\equiv0$, for $\rho$ sufficiently 
small, depending on the size of the potentials. This kind of behavior was already observed in case of mass supercritical, Sobolev subcritical nonlinearities, both with attractive and with repulsive potential 
 \cite{MR4443784,MR4304693}.

On the other hand, a decomposition with $V=V_1$ and $V_2\equiv W_2 \equiv 0$ could make  
\eqref{AS_Positive}, \eqref{AS_N4GS1} and \eqref{Assump_Neg} incompatible. 
In such a case, a local minimizer may not exist (see Remark \ref{rmk:no_neg_sol} ahead). A mountain pass solution may still exist, though. We refer to \cite[Theorem 1.1]{MR4443784} for a similar result in the Sobolev subcritical case.
\end{remark}
\begin{remark}\label{EX_po}
In  the same spirit of the previous remark, Corollary \ref{coro:coro2} is based on the fact that 
a stronger assumption which implies \eqref{def_pot} is to prescribe 
$V\in L^{r_1}(\R^{N})$ and $W\in L^{r_2}(\R^{N})$, with $\frac{N}{2}<r_1<+\infty$, $N<r_2<+\infty$. 
More precisely, under such assumptions, it is possible to obtain the decompositions in \eqref{def_pot} with $V_1$, $W_1$ as small as 
we want, so that assumptions \eqref{AS_Positive}, \eqref{AS_N4GS1} and \eqref{AS_N4GS2} follow for any such $V$, by just taking $\rho$ small enough. See Lemma \ref{lem:VinLr} below for more details.
\end{remark}
\begin{remark}\label{rmk:regmax}
One of the main aims of this contribution is to deal with low regularity assumptions on the potential 
$V$, which are of interest from the physical viewpoint (see e.g. \cite[Chapter 11]{LiebLoss}), on the line of \cite[Section 1.1]{LionCC2}. In 
particular, no assumptions are required on the (distributional) gradient of $V$. 

As a side difficulty, since $V\in L^{\frac{N}{2}}(\R^{N})+L^{\infty}(\R^{N})$, a solution of \eqref{eq:main} 
needs not to be regular enough to apply the strong maximum principle, or Harnack's inequality. Thus, non-negative solutions are 
not necessarily strictly positive. On the other hand, one can apply refined strong maximum principles, 
such as \cite[Proposition 2.1]{MR3465383}, to show that any non-negative, nontrivial solution of either 
\eqref{eq:main} or \eqref{eq:nomain2} is positive almost everywhere. Moreover, in case $V\in L^r_{\loc}
(\R^N)$, $r>N/2$, then any such a solution is strictly positive everywhere (see e.g. 
\cite[Proposition 1.3]{MR4806286}).
\end{remark}
\begin{remark}\label{rmk:lastone}
In Theorem \ref{thm:mountain_pass} we choose assumption \eqref{eq:final_ass} 
for the sake of readability. Actually, the theorem holds true under a weaker summability assumption, 
depending on the dimension $N$,  see Remark \ref{rmk:endofpaper} at the end of the paper. Also, it is 
worth mentioning that, in such assumption, $R>0$ can be arbitrarily small, as well as the 
$L^\infty(B_{2R})$ norm of $V$ may be arbitrarily large: this will just reflect on the constants $C_1,C_2$, 
not on the qualitative aspects of the result.
\end{remark}
\begin{remark}\label{rmk:comparisonwithcombined}
As we mentioned, our results are somehow analogous to those obtained in \cite{MR4096725,MR4476243,MR4433054} 
for the energy with focusing combined nonlinearities $E_{2^*,q}$, defined in \eqref{eq:E_comb_nonl}; 
in particular, in that context, our assumptions 
\eqref{AS_Positive}, \eqref{AS_N4GS1} and \eqref{AS_N4GS2} are replaced with a smallness condition 
about $b\rho^\alpha$, for a suitable exponent $\alpha>0$. The presence of a non-homogeneous, quadratic 
term, instead of an autonomous, superquadratic one, shows different phenomena, as described in Remarks 
\ref{rmk:no_neg_sol_intro}, \ref{EX_po}. This reflects in some substantial new difficulties to face, in particular 
in three aspects.  

First, in our case the Pohozaev identity \eqref{Pohozaev_mini} gives no information on the sign of 
the Lagrange multiplier $\lambda$. Such information is crucial in deducing the compactness of 
Palais-Smale sequences, either minimizing or mountain-pass ones, and it will be deduced from assumption 
\eqref{Assump_Neg}. 

The second, and more important, difficulty lies in the crucial sharp estimate of the mountain pass level 
(see Lemma \ref{upbound_moun}). In the case of combined nonlinearities, in dimension $N\ge4$, the leading term is provided by the the superquadratic perturbation $-b\int_{\mathbb{R}^{N}}|U|^{q}dx<0$. In our case, from this point of view, the perturbation is much weaker, and indefinite in sign (recall that $V$ may change sign). In some sense, our case is more similar to the (more difficult) case $N=3$ there, where the leading term 
is due to the Sobolev critical one, but here 
the terms coming from the potential $V$ may interfere with it. 

Actually, this is the reason for the technical assumption \eqref{eq:final_ass}, and for the dimensional 
restriction in Theorem \ref{thm:mountain_pass}. Indeed, assumption \eqref{eq:final_ass} is enough to 
conclude only in dimension $N\le 5$. The existence of a mountain pass solution for $N\ge6$, with suitable further assumptions on $V$, remains an interesting open problem.

Finally, it is worth mentioning that we do not assume $V$ to be radial, so we can not exploit 
symmetries and compactness for radial solutions. These tools were fundamental both in \cite{MR4096725} and in \cite{MR4433054}.
\end{remark}
\begin{remark}\label{rmk:bdd_dom}
A similar program was developed for the Sobolev critical Schr\"odinger equation on bounded domains, in 
\cite{MR3918087} (local minimizers) and \cite{MR4847285} (mountain pass solutions). In that case, 
the use of dilations is not possible, thus the Pohozaev formula just provides an inequality, which is 
not particularly useful unless the domain is star-shaped. Nonetheless, we have existence 
and multiplicity results also without such assumption, although for a set of masses which is not necessarily 
an interval.
\end{remark}

To conclude, we use the Hopf-Cole transform \cite{MR2269875} to apply our main results to the existence 
and multiplicity of solutions to ergodic Mean Field Games systems with potential and quadratic Hamiltonian. For the sake of 
simplicity we treat a particular case, although all our results can be translated in a similar fashion.
\begin{theorem}\label{thm:MFG}
Let us assume that  $V\in L^{r_1}(\R^{N})$ and $W\in L^{r_2}(\R^{N})$, with $\frac{N}{2}<r_1<+\infty$, $N<r_2<+\infty$, satisfy \eqref{Assump_Neg}. Then there exists $\alpha^*>0$ such that, for every $0<\alpha<\alpha^*$, the ergodic Mean Field Game system
\begin{equation}
\label{e:sys}
\begin{cases}
-\Delta u +\frac12|\nabla u |^2+ \lambda = -\alpha m^{\frac{2}{N-2}} + 2V(x) \qquad & \text{in} \ \R^N \smallskip\\
-\Delta m - \dive({m\nabla u })=0 \qquad  & \text{in} \ \R^N \smallskip\\
\int_{\R^N}m=1,\qquad m>0
\end{cases}
\end{equation}
admits a solution $(m,u,\lambda)$. If furthermore $3\le N\le 5$ and \eqref{eq:final_ass} holds true, 
then \eqref{e:sys} admits a second distinct solution, for every $0<\alpha<\alpha^*$, for a possibly smaller threshold $\alpha^*>0$.
\end{theorem}

While the theorem above is a simple corollary of our results, via the Hopf-Cole transform, it is an open 
problem whether it can be generalized to more general Hamiltonians, instead of $H(p)=\frac12|p|^2$. 
Starting from \cite{MR3864209}, the existence of global minimizers for ergodic Mean Field Games 
with trapping potential and general Hamiltonian has been investigated thoroughly, in the mass 
subcritical regime and up to the mass critical one \cite{MR4873723}. Local minimizers where 
obtained in the mass supercritical 
case, up to the Sobolev critical one, on bounded domains with Neumann boundary conditions in 
\cite{CCV23}. On the other hand, up to our knowledge, no result about the existence of mountain pass 
solutions for Mean Field Games with non-quadratic Hamiltonian appears in the literature.

The paper is organized as follows: Section \ref{sec:prel} contains some preliminary results, in particular concerning the decomposition of the potentials described in Remarks  \ref{rmk:no_neg_sol_intro}, \ref{EX_po}, and the proof of Corollaries \ref{coro:coro1}, \ref{coro:coro2}; in 
Section \ref{sec:MP_geo} we quickly describe the mountain pass geometry of the energy functional on the 
mass constraint, under suitable assumptions; finally, we prove our main results in Sections \ref{sec:GS} 
(Theorem \ref{Thm_neg}) and \ref{sec:mp} (Theorems \ref{thm:mountain_pass} and \ref{thm:MFG}).

\section{Notation and preliminary results}\label{sec:prel}

In this paper, for $N\ge3$, 
we denote by $L^{p}(\R^N)$ the Lebesgue space with norm $\|\cdot\|_{p}$ and by
$H^{1}(\R^N)$ the usual Sobolev space with the norm $\|\cdot\|_{H^{1}}^2=\|\nabla \cdot\|_{2}^2+\| \cdot\|_{2}^2$, while $D^{1,2}(\R^N)$ stands for the homogenous Sobolev space with norm
$\|\nabla \cdot\|_{2}$. Denoting with $S$ the best Sobolev constant of the embedding $D^{1,2}(\R^N)
\hookrightarrow L^{2^{*}}(\R^N)$, we have
\begin{equation}\label{ineq:sobolev}
S=\inf\limits_{u\in H^{1}(\R^N)\setminus \{0\}}\frac{\|\nabla u\|_2^{2}}{\|u\|_{2^{*}}^2}=\min\limits_{u\in D^{1,2}(\R^N)\setminus \{0\}}\frac{\|\nabla u\|_{2}^{2}}{\|u\|_{2^{*}}^2}.
\end{equation}

For convenience of calculation we apply the transformation
\begin{equation}\label{trans}
u= \frac{1}{\rho} U,\qquad \ \ \ \ \ \mu= \rho^{2^{*}-2}>0,
\end{equation}
to convert problem \eqref{eq:main} into the following one, which also incorporates the sign condition:
\begin{equation}\label{eq:main1}
\begin{cases}
-\Delta u +V(x)u = \lambda u + \mu |u|^{2^{*}-2}u  & \text{in }\mathbb{R}^{N},\smallskip\\
u\geq 0,\quad \int_{\mathbb{R}^{N}} u^2\,dx = 1.
\end{cases}
\end{equation}
Thus, the solutions of \eqref{eq:main1} correspond to critical points of the energy functional
\begin{equation}\label{func:main1}
E_{\mu}(u)=\frac{1}{2}\int_{\mathbb{R}^{N}}|\nabla u|^{2}dx + \frac{1}{2}\int_{\mathbb{R}^{N}}V(x) u^{2}dx-\frac{\mu}{2^{*}}\int_{\mathbb{R}^{N}}|u|^{2^{*}}dx
\end{equation}
constrained to the $L^{2}$-sphere
\begin{equation}\label{manifold:1}
\mathcal{M}:=\mathcal{M}_1=\left\{u\in H^{1}(\mathbb{R}^{N}): \int_{\mathbb{R}^{N}} u^{2}=1\right\}.
\end{equation}
Thus, we will show the existence of local minimizers, ground states and mountain pass solutions for 
$E_\mu$ over $\Mcal$, and Theorems \ref{Thm_neg}, \ref{thm:mountain_pass} will follow by the change of 
variable in \eqref{trans}.

By H\"{o}lder, Gagliardo-Nirenberg and Sobolev inequalities, we infer the following inequalities.
\begin{lemma}[{\cite[Lemma 2.2]{MR4443784}}]\label{ineq_VW}
We have, for every $u\in H^1(\R^N)$,
\begin{equation*}
\begin{aligned}\displaystyle
  \left|\int_{\mathbb{R}^{N}} V_1(x)u^{2}dx\right| &\leq  \|V_1\|_{\frac{N}{2}}\|u\|_{2^{*}}^{2} \leq S^{-1}\|V_1\|_{\frac{N}{2}}\|\nabla u\|_{2}^{2},  \\
  \left|\int_{\mathbb{R}^{N}} V_1(x)u(x)\nabla u(x)\cdot x\,dx \right| & \leq \|W_1\|_{N}\|u\|_{2^{*}}\|\nabla u\|_{2} \leq S^{-1/2} \|W_1\|_{N}\|\nabla u\|_{2}^{2},\\
  \left|\int_{\mathbb{R}^{N}} V_2(x)u^{2}dx\right| &\leq  \|V_2\|_{\infty}\|u\|_{2}^{2},  \\
  \left|\int_{\mathbb{R}^{N}} V_2(x)u(x)\nabla u(x)\cdot x\,dx \right| & \leq \|W_2\|_{\infty}\|u\|_{2}\|\nabla u\|_{2}.
\end{aligned}
\end{equation*} 
\end{lemma}

\begin{lemma}\label{lem:VW_loss}
If $(w_{n})_{n}$ is a bounded sequence in $H^{1}(\R^{N})$ with $w_{n} \rightharpoonup 0$ weakly in $H^{1}(\R^{N})$, then
\begin{equation*}
\begin{aligned}
&\int_{\R^{N}}V(x)w_{n}^{2} dx \to 0,\\
&\int_{\R^{N}}V(x)w_{n}\nabla w_{n} \cdot x\,dx \to 0,
\end{aligned}
\end{equation*}
as $n \to +\infty$.
\end{lemma}

\begin{proof}
Recall that $V$ satisfies assumption \eqref{def_pot}. We know that $w_{n}\rightharpoonup 0$ weakly in $H^{1}(\mathbb{R}^{N})$, strongly in $L^{2}_{\loc}(\mathbb{R}^{N})$ and a.e. in $\mathbb{R}^{N}$. Thus, by \eqref{def_pot}, we have
\begin{equation*}\label{Split_Po}
  \int_{\R^{N}}V(x)w_{n}^{2}dx=\int_{\R^{N}}V_{1}(x)w_{n}^{2}dx+\int_{\R^{N}}V_{2}(x)w_{n}^{2}dx,
\end{equation*}
where $V_{1}\in L^{\frac{N}{2}}(\R^{N})$ and $V_{2}\in L^{\infty}(\R^{N})$. Since $w^{2}_{n}$ is bounded in $L^{\frac{N}{N-2}}(\R^{N})$, we obtain that $w^{2}_{n} \rightharpoonup 0$ weakly in $L^{\frac{N}{N-2}}(\R^{N})$. Combining this with \eqref{def_pot} and the strong $L^{2}_{\loc}$ convergence, we have
\begin{equation*}\label{Split_2a}
\int_{\mathbb{R}^{N}}V(x) w_{n}^{2}dx\rightarrow 0, \ \ \text{as } \ n\rightarrow+\infty.
\end{equation*}

On the other hand, by the H\"{o}lder inequality, we have
\begin{equation*}
\int_{\R^{N}}V(x)w_{n}\nabla w_{n} \cdot x\,dx \leq \left(\int_{\R^{N}}(W(x)w_{n})^{2}dx\right)^{\frac{1}{2}}\|\nabla w_{n}\|_{2}.
\end{equation*}
From the assumption on $W$ in \eqref{def_pot}, proceeding similarly as before, we have
\begin{equation*}\label{Split_W}
\int_{\R^{N}}V(x)w_{n}\nabla w_{n} \cdot x\,dx \to 0, \ \ \text{as } \ n\rightarrow\infty,
\end{equation*}
completing the proof.
\end{proof}
\begin{remark}\label{rmk:vanishing_V}
Arguing in a similar way, one can show that, if  $w_{n} \rightharpoonup 0$ weakly in 
$H^{1}(\R^{N})$ and $\varphi\in C^\infty_0(\R^N)$, then
\[
\int_{\R^{N}}V(x)w_{n} \varphi \,dx \to 0
\]
as $n \to +\infty$.
\end{remark}

\begin{lemma}\label{lem:max_Princ}
Let $(u,\lambda)\in H^{1}(\R^{N})\times \R$ be a  solution of 
\eqref{eq:main1}, so that in particular $u$ is nontrivial and non-negative.
Then $\lambda<0$.
\end{lemma}

\begin{proof}
By assumption \eqref{Assump_Neg}, let $\tilde{\lambda}_{1}<0$ and $\psi_{1}\in \Mcal$, $\psi_{1}\geq 0$,  as in Remark \ref{rem:LL}. Then testing equations \eqref{eq:main1} and \eqref{eq:nomain2} with 
$\psi_{1}$ and $u$ respectively, we have
\begin{align*}
     \int_{\R^{N}} \nabla u \cdot\nabla \psi_{1}dx +   \int_{\R^{N}} V(x) u  \psi_{1}dx&=  \lambda\int_{\R^{N}} u  \psi_{1}dx +\mu\int_{\R^{N}}  |u|^{2^{*}-2}u  \psi_{1}dx,\\
     \int_{\R^{N}} \nabla u \cdot\nabla \psi_{1}dx +   \int_{\R^{N}} V(x) u  \psi_{1}dx&= \tilde{\lambda}_{1}\int_{\R^{N}}  u  \psi_{1} dx, 
\end{align*}
which means that
\begin{equation*}
  (\tilde{\lambda}_{1}-\lambda) \int_{\R^{N}}  u  \psi_{1} dx = \mu\int_{\R^{N}}  |u|^{2^{*}-2}u  \psi_{1}dx.
\end{equation*}
Recalling that both $u$ and $\psi_{1}$ are strictly positive a.e. in $\R^N$, by Remark \ref{rmk:regmax}, and that $\mu>0$, we conclude that $\lambda <  \tilde{\lambda}_{1}<0$.
\end{proof}

To analyze the lack of compactness for bounded Palais-Smale sequences of $E_\mu$, it is convenient 
to introduce the functional ``at infinity'':
\begin{equation}\label{func_without_V}
  E_{\mu,\infty}(u)=\frac{1}{2}\int_{\mathbb{R}^{N}}|\nabla u|^{2}dx-\frac{\mu}{2^{*}}\int_{\mathbb{R}^{N}}|u|^{2^{*}}dx.
\end{equation}
Moreover, we recall the following concentration-compactness result from 
\cite{Lewinlecture,MR1632171}.
\begin{lemma}\label{Concom}
Let $(u_{n})_{n}$ be a bounded sequence in $D^{1,2}(\mathbb{R}^{N})$, and define
\begin{equation*}
  m(u) = \sup \left\{\int_{\mathbb{R}^{N}} |u|^{2^{*}} : \begin{split}
  \text{there exists }(n_k)_k &\subset\N,\  
  (\alpha_k)_k\subset\R,\  
  (x_k)_k\subset\R^N,\text{ with }\\
  &\alpha_{k}^{-N/2^{*}}u_{n_{k}} \left(\frac{\cdot-x_{k}}{\alpha_{k}}\right) \rightharpoonup u
  \end{split}\right\}.
\end{equation*}
Then, $m(u)=0$ if and only if $u_{n}\rightarrow0$ strongly in $L^{2^{*}}(\mathbb{R}^{N})$.
\end{lemma}

Using Lemma \ref{Concom} and \ref{lem:max_Princ}, we have the following  result.
\begin{lemma}\label{cor:split}
Let $(u_{n},\lambda_{n})_{n}$ be a Palais-Smale sequence for $E_{\mu}$ on $\mathcal{M}$:
\begin{enumerate}
\item $(u_n)_n \subset \mathcal{M}$,
\item $E_\mu(u_n) \to c \in \R$,
\item $\|\nabla_{\mathcal M} E_\mu(u_n)\| \to 0$,
\end{enumerate}
as $n\to+\infty$, and assume that
\begin{enumerate}[resume]
\item $(u_n, \lambda_{n})_{n}$ is bounded,
\item $\|u_n^{-}\|_{H^{1}} \to 0 $.
\end{enumerate}
 Then, up to subsequences, there exist $u^{0}\in H^{1}(\R^{N})$, $u^0\ge 0$, and $\lambda^{0}\in\R$ 
 such that $u_{n}\rightharpoonup u^{0}$ weakly in $H^{1}(\R^{N})$, $\lambda_{n} \rightarrow \lambda^{0}$, and 
\[
-\Delta u^0 +V(x)u^{0}= \mu (u^{0})^{2^{*}-1} + \lambda^{0} u^0.
\]
Moreover, only one of the three following alternatives holds:
\begin{enumerate}[label=\roman*)]
\item either $\lambda^{0}<0$ and $u_{n} \to u^{0}$ strongly in $H^{1}(\R^{N})$,
\item or $\lambda^{0} <0$, $u_{n} \rightharpoonup u^{0}$ in $H^{1}(\R^{N})$ (but not strongly), and
\begin{equation}\label{eq:split_ener}
E_{\mu} (u_{n}) \geq E_{\mu}(u^{0})+\frac{1}{N}S^{N/2}\mu^{1-\frac{N}{2}}+o(1), \ \ \mbox{as} \ \ n\rightarrow +\infty,
\end{equation}
\item or $\lambda^{0} \geq 0$, and in such case $u^{0} \equiv 0$.
\end{enumerate}
\end{lemma}
\begin{proof}
Since $(u_n)_{n}$ is bounded in $H^{1}(\R^{N})$ and $\|u_n^{-}\|_{H^{1}} \to 0 $, there exists $u^{0} \geq 0$ such that (up to subsequences) $u_{n}\rightharpoonup u^{0} $ weakly in $H^{1}(\R^{N})$, 
strongly in $L^2_{\loc}(\R^N)$ and a.e. in $\R^N$. 
Moreover, since $\lambda_{n}$ is bounded, we can assume that  $\lambda_{n} \to \lambda^{0}\in\R$. By Lemma \ref{lem:max_Princ}, we know that, if $u^0\ge0$ is nontrivial, then $\lambda^{0} < 0$. Thus, we can deduce that $\lambda^{0} \geq 0$ yields that $u^{0} \equiv 0$, i.e. case iii) happens.

Hence, we can assume that $\lambda^{0} < 0$, and set $u_{1,n}=u_{n}-u^{0}$. Then, $u_{1,n}\rightharpoonup 0$ weakly in $H^{1}(\mathbb{R}^{N})$, strongly in $L^{2}_{\loc}(\mathbb{R}^{N})$ and a.e. in $\mathbb{R}^{N}$. Thus, by Lemma \ref{lem:VW_loss}, we have
\begin{equation*}\label{Split_2}
\int_{\mathbb{R}^{N}}V(x) u_{1,n}^{2}dx\rightarrow 0, \ \ \text{as } \ n\rightarrow+\infty.
\end{equation*}
By the Brezis-Lieb lemma, we have
\begin{equation}\label{BL_Eng}
  E_{\mu}(u_{n})=E_{\mu}(u^{0})+E_{\mu,\infty}(u_{1,n})+o(1),
\end{equation}
which implies that $(u_{1,n})_{n}$ is a (bounded) Palais-Smale sequence for $E_{\mu,\infty}$, defined as in \eqref{func_without_V} (recall Remark \ref{rmk:vanishing_V}). Thus, $(u_{1,n})_{n}$ satisfies
\begin{equation*}
  \|\nabla u_{1,n}\|^{2}_{2}-\lambda^{0}\|u_{1,n}\|^{2}_{2}=\mu \|u_{1,n}\|^{2^{*}}_{2^{*}}+o(1).
\end{equation*}
If $u_{1,n}\rightarrow 0$ in $H^{1}(\mathbb{R}^{N})$, we are in  case i) and the result follows. If not, we are left to 
show that case ii) holds, and we can assume that, up to subsequences, $\|u_{1,n}\|_{H^1(\R^N)}\geq d$ for a suitable constant $d>0$. Since $\lambda^{0}<0$ we obtain that 
there exists a constant $d_{1}>0$ such that, for  every $n \in \mathbb{N}$,
\begin{equation*}\label{crtical_item}
  \|u_{1,n}\|_{2^{*}}\geq d_{1}>0.
\end{equation*}
By Lemma \ref{Concom} there exist sequences $(\alpha_{n})_n$, $(x_{n})_n$, and a non-trivial $u^{1}\in H^1(\R^N)$ such that
\begin{equation*}\label{seq_dilation}
  \tilde{u}_{1,n}=\alpha_{n}^{-\frac{N-2}{2}}u_{1,n}\left(\frac{\cdot-x_{n}}{\alpha_{n}}\right)\rightharpoonup u^{1}
  \qquad\text{weakly in }H^{1}(\mathbb{R}^{N}).
\end{equation*}
By a sequence of translations we can assume $x_n=0$; moreover, direct computations show that
\begin{equation}\label{norm_dilation}
  \|\nabla \tilde{u}_{1,n}\|^{2}_{2}= \|\nabla u_{1,n}\|^{2}_{2}, \ \ \ \| \tilde{u}_{1,n}\|^{2^{*}}_{2^{*}}= \| u_{1,n}\|^{2^{*}}_{2^{*}}, \ \ \ \mbox{and } \ \| \tilde{u}_{1,n}\|^{2}_{2}= \alpha_{n}^{2}\| u_{1,n}\|^{2}_{2}.
\end{equation}
We have three different cases (up to subsequences). 

\underline{Case 1}: $\alpha_{n}\rightarrow 0$. Since $0\leq\|u_{1,n}\|^{2}_{2}\leq1$, \eqref{norm_dilation} yields
\begin{equation*}
  \| \tilde{u}_{1,n}\|^{2}_{2}= \alpha_{n}^{2}\| u_{1,n}\|^{2}_{2} \rightarrow 0, \ \ \mbox{as} \ \ {n}\rightarrow+\infty,
\end{equation*}
which means that $\tilde{u}_{1,n}\rightarrow 0$ strongly in $L^{2}(\mathbb{R}^{N})$. Therefore, we can deduce that $u^{1}=0$ a.e. in $\mathbb{R}^{N}$, which contradicts the fact that $u^{1}$ is not trivial.

\underline{Case 2}: $\alpha_{n}\rightarrow \bar{\alpha}\neq0$. We know that, for every $\varphi\in C^{\infty}_{0}(\R^N)$,
\begin{equation}\label{before_eq}
  \int_{\mathbb{R}^{N}}\nabla u_{1,n}\nabla \varphi\, dx= \mu \int_{\mathbb{R}^{N}} |u_{1,n}|^{2^{*}-2}u_{1,n} \varphi\, dx + \lambda^{0} \int_{\mathbb{R}^{N}} u_{1,n}\varphi\, dx+o\left(\|\varphi\|_{H^{1}}\right), \ \ \ \mbox{as} \ n\rightarrow +\infty.
\end{equation}
Fix $\psi(x) \in C^{\infty}_{0}(\mathbb{R}^{N})$ and define $\varphi_n(y)=\psi(\alpha_{n} y)$. Then, we have
\begin{equation*}\label{relation_varpsi}
  \|\varphi_{n}\|_{2}^{2}= \alpha_{n}^{-N}\|\psi\|^{2}_{2}, \ \ \ \|\nabla \varphi_{n}\|^{2}_{2}=\alpha_{n}^{2-N}\|\nabla\psi\|^{2}_{2}.
\end{equation*}
By \eqref{before_eq} we  deduce that
\begin{equation*}
\begin{aligned}
  &\int_{\mathbb{R}^{N}}\nabla \tilde{u}_{1,n}(x) \nabla \psi(x)dx -\mu  \int_{\mathbb{R}^{N}}|\tilde{u}_{1,n}(x)|^{2^{*}-2}\tilde{u}_{1,n}(x) \psi(x)dx\\
  =& \int_{\mathbb{R}^{N}}\alpha_{n}^{\frac{N}{2}} \nabla u_{1,n}(\frac{x}{\alpha_{n}})\cdot \nabla \psi(x)d(\frac{x}{\alpha_{n}})- \mu \int_{\mathbb{R}^{N}} \alpha_{n}^{\frac{N}{2}-1} |u_{1,n}(\frac{x}{\alpha_{n}})|^{2^{*}-2}u_{1,n}(\frac{x}{\alpha_{n}}) \psi(x)d(\frac{x}{\alpha_{n}})\\
  =& \alpha_{n}^{\frac{N}{2}-1}\left\{ \int_{\mathbb{R}^{N}}\nabla u_{1,n}(\frac{x}{\alpha_{n}})\cdot \nabla \varphi_{n}(\frac{x}{\alpha_{n}})d(\frac{x}{\alpha_{n}})- \mu \int_{\mathbb{R}^{N}}  |u_{1,n}(\frac{x}{\alpha_{n}})|^{2^{*}-2}u_{1,n}(\frac{x}{\alpha_{n}}) \varphi_{n}(\frac{x}{\alpha_{n}})d(\frac{x}{\alpha_{n}}) \right\}\\
  =&\alpha_{n}^{\frac{N}{2}-1}\left[\int_{\mathbb{R}^{N}}\nabla u_{1,n}\nabla \varphi_n\, dx- \mu \int_{\mathbb{R}^{N}} |u_{1,n}|^{2^{*}-2}u_{1,n} \varphi_n\, dx\right]\\
  =& \alpha_{n}^{\frac{N}{2}-1} \lambda^{0} \int_{\mathbb{R}^{N}} u_{1,n}\varphi_n\, dx+\alpha_{n}^{\frac{N}{2}-1}o\left(\|\varphi_n\|_{H^{1}}\right)\\ =& {\alpha_n}^{-2}\lambda^{0} \int_{\mathbb{R}^{N}} \tilde{u}_{1,n}\psi\, dx +o\left(\|\nabla \psi \|_{2} + \alpha_n^{-1}
  \|\psi \|_{2}\right).
  \end{aligned}
\end{equation*}
By weak convergence we can  pass to the limit and obtain that $u^1$ is a solution of
\[
-\Delta v = \mu |v|^{2^*-2}v + \bar{\alpha}^{-2}\lambda^{0}v\qquad\text{in }H^1(\R^N),
\] 
with $\bar{\alpha}^{-2}\lambda^{0} \neq 0$, which forces $u^{1}\equiv 0$, again a contradiction.

\underline{Case 3}: $\alpha_{n}\rightarrow + \infty$. Following the same steps as in Case 2, we find
\begin{equation*}
\begin{aligned}
\left| \int_{\mathbb{R}^{N}}\nabla \tilde{u}_{1,n} \nabla \psi\,dx -\mu  \int_{\mathbb{R}^{N}}|\tilde{u}_{1,n}|^{2^{*}-2}\tilde{u}_{1,n} \psi\,dx \right|&=\left|\alpha_{n}^{\frac{N}{2}-1} \lambda^{0} \int_{\mathbb{R}^{N}} u_{1,n}\varphi_n\, dx\right|+\alpha_{n}^{\frac{N}{2}-1}o\left(\|\varphi_n\|_{H^{1}}\right)
\\&\leq  \alpha_{n}^{\frac{N}{2}-1}\lambda^{0}\| u_{1,n}\|_{2}\|\varphi_n\|_{2}+\alpha_{n}^{\frac{N}{2}-1}o\left(\|\varphi_n\|_{H^{1}}\right)\\
 &\leq C\lambda^{0}\alpha_{n}^{-1}+o\left(\|\nabla \psi \|_{2} + \alpha_n^{-1}
  \|\psi \|_{2}\right) \rightarrow 0,
 \end{aligned}
\end{equation*}
as $\alpha_{n}\rightarrow + \infty$, 
by the H\"{o}lder inequality and $\| u_{1,n}\|_{2}\leq1$. 
Thus, we obtain that $\tilde{u}_{1,n}$ satisfies
\begin{equation}\label{eq_split}
  \int_{\mathbb{R}^{N}}|\nabla \tilde{u}_{1,n}|^{2}\,dx -\mu  \int_{\mathbb{R}^{N}}|\tilde{u}_{1,n}|^{2^{*}}\,dx=o(1),\qquad
  \text{ and }-\Delta u^1=\mu |u^1|^{2^{*}-2}u^1 .
\end{equation}
Hence, we have
\begin{equation*}
  E_{\mu,\infty}(\tilde{u}_{1,n})=\left(\frac{1}{2}-\frac{1}{2^{*}}\right)\|\nabla \tilde{u}_{1,n}\|_{2}^{2}+o(1)=\frac{1}{N}\|\nabla \tilde{u}_{1,n}\|_{2}^{2}+o(1).
\end{equation*}
Then, combining \eqref{eq_split} and the Sobolev inequality, we have
\begin{equation*}
 E_{\mu,\infty}(u_{1,n})= E_{\mu,\infty}(\tilde{u}_{1,n})=\frac{1}{N}\|\nabla \tilde{u}_{1,n}\|_{2}^{2}+o(1)\geq \frac{1}{N}\|\nabla u^{1}\|_{2}^{2}+o(1) \geq \frac{1}{N}S^{\frac{N}{2}}\mu^{1-\frac{N}{2}}+o(1),
\end{equation*}
which, together with \eqref{BL_Eng}, implies \eqref{eq:split_ener}.
\end{proof}

Finally, in the following lemma we show how a potential $V\in L^{r}(\R^{N})$ with 
$\frac{N}{2}\le r<+\infty$ can be split as $V\in L^{\frac{N}{2}}(\R^{N})+L^{\infty}(\R^{N})$, 
thus clarifying Remarks \ref{rmk:no_neg_sol_intro}, \ref{EX_po} (of course, an analogous result can 
be proved for $W\in L^r(\R^N)$ with $N\le r <+\infty$).
\begin{lemma}\label{lem:VinLr}
  Assume that $V\in L^{r}(\R^{N})$, with $\frac{N}{2}\le r< +\infty$. Then, for every $\delta>0$ there exist $V_{1}\in L^{\frac{N}{2}}(\R^{N})$ and $V_{2}\in L^{\infty}(\R^{N})$ such that
  \begin{enumerate}
  \item $\|V_{1}\|^{\frac{N}{2}}_{\frac{N}{2}}\leq 3\delta$,
  \item $V_{2}(x)\rightarrow0$ as $|x|\rightarrow \infty$.
  \end{enumerate}
\end{lemma}

\begin{proof}
We give the proof in the case $V\ge0$. In the general case, one can apply the same procedure to $V^+$ and $V^-$ separately.
 
  Case 1. If there exists $R>0$ such that $V\equiv 0$ outside $B_{R}$, we can find a constant $\eta>0$ 
  large enough such that, defining   
\[
\Omega_1 = \{x\in\R^N:V(x)\ge \eta\}, \quad \Omega_{2}= B_R \setminus  \Omega_{1}, 
\quad V_{i} = V \cdot \chi_{\Omega_{i}},
\]
then $V_{1}\in L^{\frac{N}{2}}(\R^N)$, $V_{2}\in L^{\infty}(\R^N)$, both vanishing outside $B_R$, and 
\begin{equation*}
  \|V_{1}\|^{\frac{N}{2}}_{\frac{N}{2}}\leq 3\delta
\end{equation*}
(we used $r\ge\frac{N}{2}$ and the  continuity of the measure $\mu: E \mapsto \int_E V^{N/2}dx$), and the result follows.
  
  Case 2. Assume that case 1 does not happen, and fix $\delta>0$. Using again continuity of measure, 
  there exists $R_{0}>0$ such that
\begin{equation*}
  \int_{\R^{N}\setminus \Omega_{0}}V^{r}dx\le \delta ,\ \ \ \ \mbox{with } \ \Omega_{0}=B_{R_{0}}.
\end{equation*}
Then, for $\eta_0>0$ to be chosen, define
\[
\Omega_{1,0} = \Omega_{0}\cap\{V\ge \eta_0\}, \quad \Omega_{2,0}= \Omega_0 \setminus  \Omega_{1,0}, 
\quad V_{i,0} = V \cdot \chi_{\Omega_{i,0}}.
\]
Then $V_{1,0}\in L^{\frac{N}{2}}(\R^N)$, $V_{2,0}\in L^{\infty}(\R^N)$, and, taking $\eta_{0}$ large enough,
\begin{equation*}
  \|V_{1,0}\|^{\frac{N}{2}}_{\frac{N}{2}}\leq \delta,\qquad \|V_{2,0}\|_{\infty}\leq \eta_0.
\end{equation*}
Then, we have that there exists $R_{1}$ such that
\begin{equation*}
  \int_{\Omega_{1}}V^{r}dx\le \frac{\delta}{2} \ \ \ \ \mbox{with } \ \Omega_{1}=B_{R_{1}}\setminus B_{R_{0}}.
\end{equation*}
Taking $\eta_{1}=1$, we define
\[
\Omega_{1,1} = \Omega_{1}\cap\{V\ge \eta_1\}, \quad \Omega_{2,1}= \Omega_1 \setminus  \Omega_{1,1}, 
\quad V_{i,1} = V \cdot \chi_{\Omega_{i,1}}.
\]
Notice that, since $V\ge 1$ on $\Omega_{1,1}$, 
\begin{equation*}
  \|V_{1,1}\|^{\frac{N}{2}}_{\frac{N}{2}}=\int_{\Omega_{1,1}} V^{\frac{N}{2}}dx\leq \int_{\Omega_{1,1}} V^{r}dx\leq \frac{\delta}{2} ,\qquad \|V_{2,1}\|_{\infty}\leq 1.
\end{equation*}
Recursively, for $k\geq2$, we can obtain that there exists $R_{k}$ such that
\begin{equation*}
  \int_{\Omega_{k}}V^{r}dx\le\frac{\delta}{2^{k}} \ \ \ \ \mbox{with } \ \Omega_{k}=B_{R_{k}}\setminus B_{R_{k-1}}.
\end{equation*}
Taking $\eta_{k}^{r-\frac{N}{2}}=\frac{1}{k}$ when $r>\frac{N}{2}$, or simply $\eta_k\to0^+$ if 
$r=\frac{N}{2}$, we define
\[
\Omega_{1,k} = \Omega_{k}\cap\{V\ge \eta_k\}, \quad \Omega_{2,k}= \Omega_k \setminus  \Omega_{1,k}, 
\quad V_{i,k} = V \cdot \chi_{\Omega_{i,k}}.
\]
By the definition of $\eta_k$ we have 
\begin{equation*}
  \|V_{1,k}\|^{\frac{N}{2}}_{\frac{N}{2}}=\int_{\Omega_{1,k}} V^{\frac{N}{2}}dx\leq
  \int_{\Omega_{1,k}} V^r\cdot \eta_k^{\frac{N}{2}-r}dx\leq k \int_{\Omega_{1,k}} V^{r}dx\leq k\frac{\delta}{2^k} ,\qquad \|V_{2,k}\|_{\infty}\leq \eta_k.
\end{equation*}

Finally, the sets $\Omega_{i,k}$ are pairwise disjoint and, since $R_{k}\rightarrow +\infty$ as $k \rightarrow +\infty$ (otherwise $V\equiv 0$ outside some ball), their union is $\R^N$. We infer
\begin{equation*}
  V_{2}=\sum_{k}V_{2,k}\in L^{\infty}(\R^{N}),\ \ \ \mbox{with } V_{2}(x)\rightarrow0 \ \ \mbox{as } \ |x|\rightarrow \infty
\end{equation*}
(since $\eta_k\to0$ as $k\to+\infty$), 
and $V_{1}=\sum_{k} V_{1,k}$ is such that
\begin{equation*}
  \int_{\R^{N}}V_{1}^{\frac{N}{2}}dx=\sum_{k}\int_{\Omega_{k}}V_{1,k}^{\frac{N}{2}}dx \leq \delta +\delta \sum_{k=1}^{+\infty}k2^{-k}=3\delta.\qedhere
\end{equation*}
\end{proof}

\begin{proof}[Proofs of Corollaries \ref{coro:coro1}, \ref{coro:coro2}]
To show that the corollaries are implied by Theorem \ref{Thm_neg}, we have to show that the assumptions 
therein imply the validity of  \eqref{AS_Positive}, \eqref{AS_N4GS1} and \eqref{AS_N4GS2}, for 
$\rho$ sufficiently small. To this aim, it is enough to obtain the decomposition in \eqref{def_pot}, with
$\|V_1\|_{\frac{N}{2}}$, $\|W_1\|_{N}$ sufficiently small. In the case of Corollary  \ref{coro:coro2} 
this immediately follows by applying twice Lemma \ref{lem:VinLr} to $V\in L^{r_1}(\R^{N})$, 
$W\in L^{r_2}(\R^{N})$, respectively, choosing $\delta$ small enough. On the other hand, for Corollary  
\ref{coro:coro1}, starting from a first decomposition  in \eqref{def_pot}, we can apply 
Lemma \ref{lem:VinLr} to $V_1\in L^{N/2}(\R^{N})$, $W_1\in L^{N}(\R^{N})$, again with $\delta$ 
sufficiently small. In this way, we can write $V_1=V_{1,1} + V_{1,2}$ and
\[
V= V_1 + V_2 = \tilde V_{1} + \tilde V_2,\qquad\text{with }\tilde V_1 = V_{1,1} \text{ and }
\tilde V_2 = V_{1,2} + V_2
\]
(and perform a similar decomposition for $W$).
\end{proof}

\section{Mountain pass geometry}\label{sec:MP_geo}

In this section we show that, recalling definitions \eqref{func:main1}, \eqref{manifold:1}, the energy 
functional $E_\mu$ has a mountain pass structure on the $L^2$ sphere $\Mcal$, at least under additional assumptions (beyond  \eqref{def_pot}, \eqref{Assump_Neg}). This will lead to the definition of candidate 
local minimum level $m_\mu$ and mountain pass level $c_\mu$.

Following \cite{MR1430506,MR4304693,MR4443784}, for every $u \in \mathcal{M}$ and $h > 0$, we define the scaling
\begin{equation*}
  u_{h}(x)= h\star u(x)= h^{\frac{N}{2}}u(hx),
\end{equation*}
so that $u_{h}\in \mathcal{M}$ for every $h>0$. Thus,
\begin{equation}\label{func_sca}
\begin{split}
  E_{\mu}(u_{h})&=\frac{h^{2}}{2}\int_{\mathbb{R}^{N}}|\nabla u|^{2}dx +\frac{1}{2}\int_{\mathbb{R}^{N}}V\left(x\right)  u_h^{2}\,dx-\frac{\mu h^{2^{*}}}{2^{*}}\int_{\mathbb{R}^{N}}|u|^{2^{*}}dx\\&=\frac{h^{2}}{2}\int_{\mathbb{R}^{N}}|\nabla u|^{2}dx +\frac{1}{2}\int_{\mathbb{R}^{N}}V\left(\frac{x}{h}\right) u^{2}dx-\frac{\mu h^{2^{*}}}{2^{*}}\int_{\mathbb{R}^{N}}|u|^{2^{*}}dx.
\end{split}
\end{equation}

Using Lemma \ref{ineq_VW},  we can show that $E_{\mu}$ possesses a mountain-pass geometric structure.

\begin{lemma}\label{geo}
Let $u \in \mathcal{M}$ be fixed. Then, we have
\begin{enumerate}
  \item\label{geo_1} $\|\nabla u_{h}\|_{2} \rightarrow 0$ and $E_{\mu}(u_{h}) \rightarrow 0$ as $h \rightarrow 0$,
  \item\label{geo_2} $\|\nabla u_{h}\|_{2} \rightarrow + \infty$ and $E_{\mu}(u_{h}) \rightarrow - \infty$ as $h \rightarrow + \infty$.
\end{enumerate}
\end{lemma}
\begin{proof}
By Lemma \ref{ineq_VW} and \eqref{def_pot}, we have, for every $\eps>0$ there exist $R_\eps$ and $\bar{h}_{\eps}$ such that
\begin{equation*}
\|V_{2}\|_{L^{\infty}(\R^{N}\setminus B_{R_\eps})} \leq \eps,
\end{equation*}  
and, if $h \leq \bar{h}_{\eps}$,
\begin{equation*}
   \left|\int_{\mathbb{R}^{N}} V(x) u_{h}^{2}dx\right| \leq \left|\int_{\R^{N}}V_{1}u^{2}_{h}dx\right|+\left|\int_{B_{R_{\eps}}}V_{2}u^{2}_{h}dx\right|+\left|\int_{\R^{N}\setminus B_{R_{\eps}}}V_{2}u^{2}_{h}dx\right| \leq 3\eps.
\end{equation*}
Combining this with \eqref{func_sca}, we can conclude that \ref{geo_1} holds.

On the other hand, by \eqref{func_sca}, we infer that
\begin{equation*}
  E_{\mu}(u_{h})\leq \frac{1}{2}h^{2}\|\nabla u\|^{2}_{2}\left(1+S^{-1}\|V_{1}^{+}\|_{\frac{N}{2}}\right)+\frac{1}{2}\|V_{2}^{+}\|_{\infty}-\frac{\mu h^{2^{*}}}{2^{*}}\|u\|_{2^{*}}^{2^{*}},
\end{equation*}
which yields \ref{geo_2}.
\end{proof}

To proceed, we provide a lower estimate of $E_{\mu}$. Recall that $S$ is the best constant 
in the Sobolev embedding \eqref{ineq:sobolev}.
\begin{lemma}\label{est_inf}
Let us define
\[
\bar{R}=\bar{R}_\mu:=S^{\frac{N}{4}}\mu^{\frac{1}{2}-\frac{N}{4}},
\] 
and let us assume that assumption \eqref{AS_Positive} holds with $C_0$ such that 
\begin{equation}\label{eq:C0}
0<C_0<\frac{2S^{\frac{N}{2}}}{N\left(S^{\frac{N}{2}-1}+1\right)}.
\end{equation}
Then there exists $\eps>0$ small such that
\begin{equation*}
  E_{*}=\inf\{E_{\mu}:  \|u\|^{2}_{2}\le1, \ (1-\eps)\bar R^2\le \|\nabla u\|_{2}^2\le \bar{R}^2\}>0.
\end{equation*}
\end{lemma}
\begin{proof}
Let $u\in H^1(\R^N)$ satisfy
\[
\|u\|^{2}_{2}\le1,\qquad \ (1-\eps)\bar R^2\le \|\nabla u\|_{2}^2\le \bar{R}^2
\]
with $\eps>0$ small to be fixed. By Lemma \ref{ineq_VW} 
we have
 \begin{equation*}
 \begin{aligned}
   E_{\mu}(u)&\geq \frac{1}{2}\|\nabla u\|^{2}_{2}-\frac{\mu}{2^{*}}\| u\|_{2^*}^{2^{*}}-\frac{1}{2}\int_{\R^{N}}\left[V^{-}_{1}+V^{-}_{2}\right]u^{2}dx\\&\geq \frac{1}{2}\|\nabla u\|^{2}_{2}-\frac{\mu}{2^{*}}S^{-\frac{N}{N-2}}\|\nabla u\|_{2}^{2^{*}}-\frac{1}{2} S^{-1}\|V_{1}^{-}\|_{\frac{N}{2}}\|\nabla u\|_{2}^2-\frac{1}{2}\|V_{2}^{-}\|_{\infty}\| u\|_{2}^2\\
   & \geq\left(\frac{1}{N}- \frac{\eps}{2}\right)S^{\frac{N}{2}}\mu^{1-\frac{N}{2}}-\frac{1}{2} S^{\frac{N}{2}-1}\|V_{1}^{-}\|_{\frac{N}{2}}\mu^{1-\frac{N}{2}}-\frac{1}{2}\|V_{2}^{-}\|_{\infty} .
    \end{aligned}
 \end{equation*} 
 Now, by \eqref{AS_Positive} and \eqref{trans}, we have
 \begin{equation*}
   \frac{1}{2} S^{\frac{N}{2}-1}\|V_{1}^{-}\|_{\frac{N}{2}}+\frac{1}{2}\mu^{\frac{N}{2}-1}\|V_{2}^{-}\|_{\infty}\le \frac{C_0}{2}\left(S^{\frac{N}{2}-1}+1\right)\le \left(\frac{1}{N}- \frac{\eps}{2}\right)S^{\frac{N}{2}}-\delta,
 \end{equation*}
 for small positive constants $\delta$, $\eps$ (notice that this is possible since $C_0$ 
 satisfies \eqref{eq:C0}).  This means
 \begin{equation*}
  E_{\mu}(u)\geq \delta\mu^{1-\frac{N}{2}}>0,
\end{equation*}
concluding the proof.
\end{proof}

Based on the previous lemma, we define
\begin{equation}\label{level_min}
  m_{\mu}:=\inf\{E_{\mu}:\ u\in \mathcal{M},\|\nabla u\|_{2}\leq \bar{R}\}.
\end{equation}
In particular, if $m_{\mu}<E_*$, then it is a candidate local minimum level.
Now, by Lemmas \ref{geo} and \ref{est_inf}, we fix $w \in \mathcal{M}$ with 
\begin{equation*}
 \|\nabla w\|^{2}_{2}=\bar{R}^{2}= S^{\frac{N}{2}}\mu^{1-\frac{N}{2}}, 
\end{equation*}
thus, there exist $0<h_{0}<1<h_{1}$ such that
\begin{equation}\label{eq:Rbar}
\begin{split}
  &\|\nabla w_{h_{0}}\|_{2}< \bar{R}, \ \  \ \  E_{\mu}(w_{h_{0}})< E_{*}, \\
  &\|\nabla w_{h_{1}}\|_{2}> \bar{R}, \ \  \ \  E_{\mu}(w_{h_{1}})< 0.
\end{split}
\end{equation}
Then, we define in a standard way the candidate mountain pass value
\begin{equation}\label{level_Moun}
  c_{\mu}:=\inf_{\gamma \in\Gamma}\sup_{t\in[0,1]}E_{\mu}(\gamma(t)), \ \ \ \ \mbox{where } \Gamma:=\{\gamma\in C([0,1],\mathcal{M}),\gamma(0)=w_{h_{0}},\gamma(1)=w_{h_{1}}\}.
\end{equation}
Notice that, by Lemma \ref{est_inf},
\begin{equation}\label{eq:positive_c_mu}
c_\mu\ge E_*>0.
\end{equation}


\section{Ground state solution (proof of Theorem \ref{Thm_neg})}\label{sec:GS}

This section is devoted to the proof of Theorem \ref{Thm_neg}, or, better, to a version of such theorem 
involving $E_\mu$ on $\Mcal$, via the change of variables \eqref{trans}. Throughout the section, we assume that assumptions \eqref{def_pot}, \eqref{Assump_Neg} and \eqref{AS_Positive} hold true. We first aim to prove the existence of the local minimizer. Then, under further assumptions, we prove that the local minimizer is the ground state, according to Definition \ref{def:GS}. To begin with, we obtain the following upper bound estimate for 
$m_{\mu}$, defined as in \eqref{level_min}.
\begin{lemma}\label{lem:negative_m}
  Under assumption \eqref{Assump_Neg} we have $m_{\mu}<0$.
\end{lemma}
\begin{proof}
By \eqref{Assump_Neg}, following the same steps as in \cite[Section 4.1]{MR4443784}, we can obtain the result.

 Indeed, let $\psi_1 \in \mathcal{M}$, $\tilde{\lambda}_{1}<0$ be as in Remark \ref{rem:LL}. Thus, for every $t>0$, we have
\begin{equation*}
E_{\mu}(t \psi_1) \leq \frac12 \tilde{\lambda}_{1}-\frac{t^{2^{*}}\mu}{2^{*}}\|\psi_1\|^{2^{*}}_{2^{*}}<\frac12 \tilde{\lambda}_{1}<0.
\end{equation*}
Taking $\bar{t}=\frac{\bar{R}}{\|\nabla \psi_1\|_{2}}$, with $\bar R$ as in Lemma \ref{est_inf}, we have $\|\nabla \bar{t} \psi_1\|_{2}=\bar{R}$ and $E_{\mu}(\bar{t} \psi_1) <0$, which, combined with   
Lemma \ref{est_inf}, yields that
\begin{equation*}
\|\bar{t}\psi_1\|_{2}=\bar{t}>1.
\end{equation*}
Then, by the definition of $\bar{t}$, we obtain that 
\[
\|\nabla \psi_1\|_{2}= \frac{\bar{R}}{\bar{t}}<\bar{R}.
\]
Thus, since $\psi_1 \in \mathcal{M}$ and $\|\nabla \psi_1\|_{2}< \bar{R}$, we have
\begin{equation*}
m_{\mu}\leq E_{\mu}(\psi_1)<0. \qedhere
\end{equation*} 
\end{proof}

Next we are going to show that $m_{\mu}$ is a local minimum level. 
\begin{lemma}\label{lem:loc_minimizer}
If \eqref{Assump_Neg} and \eqref{AS_Positive} hold true, with $C_0$ as in Lemma \ref{est_inf}, then $m_\mu$ is achieved.
\end{lemma}
\begin{proof}
Let $(v_{n})_{n}$ be a minimizing sequence for $m_\mu$. Since $E_{\mu}$ and $\mathcal{M}$ are even, we can assume that $v_{n}\geq 0$. Moreover, by Lemma \ref{est_inf}, we know $\|\nabla v_{n}\|_{2}^2 \le (1-\eps)\bar{R}^2$ for every $n$. Then, by Ekeland's principle, we can find a Palais-Smale sequence $(u_{n})_{n}$ for $E_{\mu}$ constrained to 
$\mathcal{M}$ at level $m_{\mu}$ such that
\begin{equation*}
  \|u_{n}-v_{n}\|_{H^{1}} \to 0,
\end{equation*}
as $n \to +\infty$. In particular, as $n\to +\infty$,
\begin{equation*}\label{PS_Mini}
  E_{\mu}(u_{n})\to m_{\mu},\qquad \nabla_{\mathcal{M}}E_{\mu}(u_{n}) \to 0,\qquad \|u_{n}^-\|_{H^{1}} \to 0.
  \end{equation*}
 Also, since
$(u_{n})_{n}$ is bounded, 
\begin{equation*}\label{bound_Mlad}
  \lambda_{n}:= \int_{\mathbb{R}^{N}}|\nabla u_n|^{2}dx + \int_{\mathbb{R}^{N}}V(x) u_n^{2}dx-
  \mu\int_{\mathbb{R}^{N}}|u_n|^{2^{*}}dx\ \ \  \text{is bounded too.}
\end{equation*}
Therefore we are in a position to apply Lemma \ref{cor:split}, obtaining that, up to subsequences, $u_{n}\rightharpoonup u^{0} \ge 0$ weakly in $H^{1}(\mathbb{R}^{n})$, strongly in $L^2_{\loc}(\R^N)$ and almost everywhere, $\lambda_{n} \to \lambda^{0}\in \mathbb{R}$, and $u^0$ weakly solves
\[
-\Delta u^0 +V(x)u^{0}= \mu(u^{0})^{2^{*}-1} +  \lambda^{0} u^0.
\] 
Moreover, by weak lower semicontinuity of the norm, $\| u^0\|_{2} \leq 1$ and $\|\nabla u^0\|_{2} 
< \bar{R}$.

Now, we claim that $\lambda^{0} < 0$. Indeed, if $u^0\ge0$ is nontrivial then this is the case by Lemma \ref{lem:max_Princ}. On the other hand, let $u^0 \equiv 0$. In this case, Lemma \ref{lem:VW_loss} implies that, as $n\rightarrow+\infty$,
\[
\int_{\mathbb{R}^{N}}V(x) u_{n}^{2}dx\rightarrow 0,
\]
so that Lemma \ref{lem:negative_m} yields
\[
\begin{split}
0> m_\mu + o(1) &=  \frac12\int_{\mathbb{R}^{N}}|\nabla u_n|^{2}dx + \frac12\int_{\mathbb{R}^{N}}V(x) u_n^{2}dx-
  \frac{\mu}{2^*}\int_{\mathbb{R}^{N}}|u_n|^{2^{*}}dx\\
  &= \frac1N\int_{\mathbb{R}^{N}}|\nabla u_n|^{2}dx + \frac1N\int_{\mathbb{R}^{N}}V(x) u_n^{2}dx 
  +\frac{1}{2^*}\lambda_n\\
  &\ge \frac{1}{2^*}\lambda^{0} + o(1),
\end{split}
\]
and $\lambda^{0}<0$ also in this case.

To conclude the proof, we are going to show that $u_n\to u^0$ strongly in $\Mcal$, so that $m_{\mu}$ 
is achieved. We argue by contradiction, assuming that the convergence is only weak. 
Since $\lambda^{0}<0$ we are in case (ii) of Lemma \ref{cor:split}, and we infer
\begin{equation}\label{ener_minsplit}
  0 > m_{\mu} \geq E_{\mu}(u^{0})+\frac{1}{N}S^{N/2}\mu^{1-\frac{N}{2}},
\end{equation}
and, in particular, $u^0$ is not trivial. 
Let $a:=\|u^{0}\|_{2}$. If $a=1$, by definition of $m_{\mu}$, we have $E_{\mu}(u^{0})\geq m_{\mu}$, which contradicts \eqref{ener_minsplit}. 

On the other hand, if $0<a<1$, by Lemma \ref{est_inf}, we claim that $\|\nabla \frac{u^{0}}{a}\|_{2}< \bar{R}$. If not, we can assume that there exists $\bar{t} \in (1,\frac{1}{a})$ such that $\|\nabla \bar{t} u^{0}\|_{2}=\bar{R}$, $\|\bar{t}u^{0}\|_{2}<1$ and $E_{\mu}(\bar{t}u^{0})<\bar{t}^{2}E_{\mu}(u^{0})<0$. However, by Lemma \ref{est_inf}, we know $E_{\mu}(\bar{t}u^{0})>0$, which results in a contradiction. From $\|\nabla \frac{u^{0}}{a}\|_{2}< \bar{R}$ and the definition of $m_{\mu}$, we obtain $E_{\mu}(\frac{u^{0}}{a})\geq m_{\mu}$. Also, we have
\begin{equation*}
  E_{\mu}\left(\frac{u^{0}}{a}\right)\leq \frac{1}{a^{2}}E_{\mu}(u^{0})\leq E_{\mu}(u^{0}),
\end{equation*}
which yields that $E_{\mu}(u^{0})\geq m_{\mu}$, again a contradiction with \eqref{ener_minsplit}.
\end{proof}

To complete the proof of Theorem \ref{Thm_neg} we are left to show that, under the further assumptions therein, a local minimizer is a ground state. In fact, for such problems, proving that the ground state is the local minimizer is not straightforward. We derive this conclusion by estimating the $H^1$-norm of all negative-energy solutions. 

To begin with, by assumptions \eqref{AS_N4GS1}, \eqref{AS_N4GS2} and $\mu^{\frac{N}{2}-1}=\rho^{2}$, we have
\begin{equation}\label{GS_AS1}
   \max\left\{ \| \left[(N-4)V_{1} \right]^{+}\|_{\frac{N}{2}},  \|W_{1}\|_{N}\right\}\leq C_{1},
\end{equation}
and
\begin{equation}\label{GS_AS11}
   \mu^{\frac{N}{2}-1}\max\left\{ \| \left [(N-4)V_{2} \right]^{+}\|_{\infty},  \|W_{2}\|^{2}_{\infty}\right\}\leq C_{2},
\end{equation}
where $C_{1},C_{2}$ are mutually independent positive constants to be chosen below. Actually,  in different dimensions, the assumptions have the following three cases: For $N \geq 5$,
\[
   \max\left\{ \| (N-4)V_{1}^{+}\|_{\frac{N}{2}},  \|W_{1}\|_{N}\right\}\leq C_{1}, \ \ \text{and} \ \ \ \mu^{\frac{N}{2}-1}\max\left\{ \| (N-4)V_{2} ^{+}\|_{\infty},  \|W_{2}\|^{2}_{\infty}\right\}\leq C_{2}.
\]
 For $N =4$,
\[
    \|W_{1}\|_{N}\leq C_{1}, \ \ \text{and} \ \ \ \mu^{\frac{N}{2}-1}  \|W_{2}\|^{2}_{\infty} \leq C_{2}.
\]
For $N = 3$,
\[
   \max\left\{ \| V_{1}^{-}\|_{\frac{3}{2}},  \|W_{1}\|_{3}\right\}\leq C_{1}, \ \ \text{and} \ \ \ \mu^{\frac{1}{2}}\max\left\{ \| V_{2} ^{-}\|_{\infty},  \|W_{2}\|^{2}_{\infty}\right\}\leq C_{2}.
\]

Thus, we have the following lemma to obtain the conclusion.
\begin{lemma}\label{Comp_LCGS}
  Assume \eqref{AS_N4GS1} and \eqref{AS_N4GS2} hold, for a suitable choice of $C_1$, $C_2$. If $u$ 
  solves
\begin{equation}\label{eq:pb_with_smaller_a}
\begin{cases}
-\Delta u +V(x)u = \mu |u|^{2^{*}-2}u + \lambda u & \text{in }\mathbb{R}^{N},\smallskip\\
\int_{\mathbb{R}^{N}} u^2\,dx = a^2\le 1,
\end{cases}
\end{equation}
with $E_{\mu}(u) < 0$, then $\|\nabla u\|_{2} < \bar{R}$, where $\bar R=S^{\frac{N}{4}}\mu^{\frac{1}{2}-\frac{N}{4}}$. 
\end{lemma}
\begin{proof}
  Since $u$ solves \eqref{eq:pb_with_smaller_a}, we obtain that $u$ satisfies the following Pohozaev identity
  \begin{equation}\label{Pohozaev_mini}
  \|\nabla u\|_{2}^{2}-\mu\|u\|_{2^{*}}^{2^{*}}+\frac{1}{2}\int_{\mathbb{R}^{N}}V(x)(Nu^{2}+2u \nabla u \cdot x)dx=0.
\end{equation}
Then, combining \eqref{Pohozaev_mini} and $E_{\mu}(u)<0$, which implies $a>0$, we deduce that
\begin{equation*}
\begin{aligned}\displaystyle
\frac{1}{N}\int_{\R^{N}}|\nabla u|^{2}dx&<\frac{N-4}{4}\int_{\R^{N}}V(x)u^{2}dx+\frac{N-2}{2N}\int_{\R^{N}}V(x)u\nabla u\cdot x dx\\
&\le\frac{1}{4}\int_{\R^{N}}\left[(N-4)V(x)\right]^{+}u^{2}dx+\frac{N-2}{2N}\int_{\R^{N}}V(x)u\nabla u\cdot x dx.
\end{aligned}
\end{equation*}
Set $V^{*}_{i}:=\left[(N-4)V_{i}(x)\right]^{+}$, where $i=1,2$. From this, since 
$\|u\|_2=a\le1$, by Lemma \ref{ineq_VW}, we have
\begin{equation*}
  \|\nabla u\|^{2}_{2}< A(N)\left(S^{-1}\|V^{*}_{1}\|_{\frac{N}{2}}\|\nabla u\|_{2}^{2}+\|V^{*}_{2}\|_{\infty} \right)+B(N)\left(S^{-\frac{1}{2}}\|W_{1}\|_{N}\|\nabla u\|^{2}_{2}+\|W_{2}\|_{\infty}\|\nabla u\|_{2}\right),
\end{equation*}
where $A(N)=\frac{N}{4}$ and $B(N)=\frac{N-2}{2}$. Thus, we have
\begin{equation*}
  \left(1-A(N)S^{-1}\|V^{*}_{1}\|_{\frac{N}{2}}-B(N)S^{-\frac{1}{2}}\|W_{1}\|_{N}\right)\|\nabla u\|^{2}_{2}< A(N)\|V^{*}_{2}\|_{\infty}+B(N)\|W_{2}\|_{\infty}\|\nabla u\|_{2}.
\end{equation*}
By \eqref{GS_AS1}, choosing an appropriate positive constant $C_{1}$, we have
\begin{equation*}
  1-A(N)S^{-1}\|V^{*}_{1}\|_{\frac{N}{2}}-B(N)S^{-\frac{1}{2}}\|W_{1}\|_{N} \geq 1-A(N)S^{-1}C_{1}-B(N)S^{-\frac{1}{2}}C_{1}>\delta_{1}>0,
\end{equation*}
where $\delta_{1}$ is a positive constant depending on $V_{1}$ and $W_{1}$. In particular, 
$\delta_{1}$ is independent of both $V_{2}$ and $W_{2}$. Therefore, we obtain
\begin{equation}\label{eq:eqforrmk}
 \delta_{1}\|\nabla u\|^{2}_{2}< A(N)\|V^{*}_{2}\|_{\infty}+B(N)\|W_{2}\|_{\infty}\|\nabla u\|_{2}.
\end{equation}
Thus, we deduce that
\begin{equation*}
\|\nabla u\|^{2}_{2}< \frac{2}{\delta_{1}} \max\left\{A(N)\|V^{*}_{2}\|_{\infty},B(N)\|W_{2}\|_{\infty}\|\nabla u\|_{2}\right\}.
\end{equation*}
By \eqref{GS_AS11}, we have
\begin{equation*}
\|\nabla u\|^{2}_{2}< C(N,\delta_{1}) \max\left\{\|V^{*}_{2}\|_{\infty},\|W_{2}\|_{\infty}^{2}\right\} \leq C(N,\delta_{1}) C_{2}\mu^{1-\frac{N}{2}}=\frac{C(N,\delta_{1})}{S^{\frac{N}{2}}} C_{2} \bar R^2
<\bar R^2,
\end{equation*}
for an appropriate choice of the constant $ C_{2}$.
\end{proof}
\begin{corollary}\label{coro:GS}
Under the assumptions of the previous lemma, consider any solution $u\in H^1(\R^N)$ of
\eqref{eq:pb_with_smaller_a}. Then 
\[
E_\mu(u) \ge m_\mu.
\]
\end{corollary}
\begin{proof}
If $E_\mu(u)\ge0$ (e.g. if $a=0$) the result is trivial. If $a=1$, the result follows from the very definition of $m_\mu$, together with Lemma \ref{Comp_LCGS}. Finally, if $E_\mu(u)<0$ and $0<a<1$, we can reduce 
to the case $a=1$ by considering $u/a$ instead of $u$, as in the end of the proof of Lemma \ref{lem:loc_minimizer}. 
\end{proof}

Finally, from the above lemma, we can conclude that all local minimizers at level $m_\mu<0$ are ground state solutions.

\begin{remark}\label{rmk:no_neg_sol}
Arguing as in Lemma \ref{Comp_LCGS} we obtain that, if $V_{2}\equiv 0$ (but $V=V_1$ still satisfies  assumptions \eqref{def_pot} and \eqref{AS_Positive}, with $C_0$ as above), 
the equation has no solutions at negative energy levels. This is because, if such a solution existed, 
$V_{2},W_{2}=0$ would yield $\delta\|\nabla u\|_{2}^{2}<0$ in \eqref{eq:eqforrmk}, a contradiction. This implies that, 
in such a case, the infimum in \eqref{Assump_Neg} would not be negative, so a local minimizer can not 
exist. On the other hand, a mountain pass solution may still exist, as observed in 
Remark \ref{rmk:no_neg_sol_intro}. 
\end{remark}

\section{Mountain pass solution (proof of Theorem \ref{thm:mountain_pass})}\label{sec:mp}

This section is devoted to the proof of Theorem \ref{thm:mountain_pass}. Throughout this section, 
we assume that $N=3,4,5$ and that all the previous assumptions hold. Moreover, we also assume 
\eqref{eq:final_ass}, for some $R>0$. Using the mountain pass geometry which we analyzed in Section 
\ref{sec:MP_geo}, see \eqref{level_Moun}, we can construct an associated bounded Palais-Smale sequence by adopting 
the classical strategy introduced by Jeanjean in \cite{MR1430506}. More precisely, define:
\begin{equation*}
  \tilde{E}_{\mu}(u,s):=E_{\mu}(e^{\frac{N}{2}s}u(e^{s}x)), \ \ \ \ \mbox{for all }(u,s)\in H^{1}(\R^{N})\times \R,
\end{equation*}
\begin{equation*}
  \tilde{\Gamma}:=\{\tilde{\gamma}\in C([0,1],\mathcal{M}\times \R):\tilde{\gamma}(0)=(w_{h_{0}},0),\tilde{\gamma}(1)=(w_{h_{1}},0)\}
\end{equation*}
and
\begin{equation*}
  \tilde{c}_{\mu}:=\inf_{\tilde{\gamma} \in\tilde{\Gamma}}\sup_{t\in[0,1]}\tilde{E}_{\mu}(\tilde{\gamma}(t))
\end{equation*}
(recall the definition of $w_{h_i}$ in \eqref{eq:Rbar}).

\begin{lemma}\label{ener_compare}
\begin{equation*}
  \tilde{c}_{\mu}=c_{\mu}. 
\end{equation*}
\end{lemma}
\begin{proof}
  For any $\gamma\in \Gamma$, by the definition of $\tilde{\Gamma}$, we have $(\gamma,0)\in \tilde{\Gamma}$, which implies $\Gamma \subset \tilde{\Gamma}$. Thus, we obtain $\tilde{c}_{\mu}\leq c_{\mu}$. Conversely, for any $\tilde{\gamma}=(u,s)\in \tilde{\Gamma}$, then $\gamma:=
  e^{\frac{N}{2}s}u(e^{s}\cdot)\in \Gamma$ is such that
  \begin{equation*}
    \max_{t\in[0,1]}E_{\mu}(\gamma(t))=\max_{t\in[0,1]}\tilde{E}_{\mu}(\tilde{\gamma}(t)),
  \end{equation*}
  which means $\tilde{c}_{\mu}\geq c_{\mu}$.
\end{proof}

\begin{remark}\label{equal_twopath}
An immediate consequence of Lemma \ref{ener_compare} is that, recalling Lemma \ref{est_inf}, we have 
$\tilde{c}_{\mu}\geq E_{*}>0$. Also, if $(u_{n},s_{n})_{n}$ is a $(PS)_{c}$ sequence for 
$\tilde{E}_{\mu}$, there exists $\tilde{u}_{n}=(u_{n},s_{n})$ such that $(\tilde{u}_{n})_{n}$ is a 
$(PS)_{c}$ sequence for $E_{\mu}$. Meanwhile, if $(u_{n})_{n}$ is a $(PS)_{c}$ sequence for $E_{\mu}$, then $(u_{n},0)_{n}$ is a $(PS)_{c}$ sequence for $\tilde{E}_{\mu}$.
\end{remark}

Subsequently, we aim to estimate $c_{\mu}$ to find its upper bound. Before that, we present the following lemma, which states that if $V \in L^{p}(\R^{N})$ for $1\leq p<+\infty$, under small scale perturbations, the variation of $V$ in $L^{p}(\R^{N})$ can be considered negligible.

\begin{lemma}\label{Cont_V}
  Assume that $V \in L^{p}(\R^{N})$ with $1\leq p< +\infty$, $\alpha \in \R$ and $s>0$. Then,
  \begin{equation*}
   \left( \int_{\R^{N}}\left[ s^{\alpha}V\left(\frac{x}{s}\right)-V(x) \right]^{p}dx \right)^{\frac{1}{p}} \to 0, \ \ \ \mbox{as } \ s \to 1.
  \end{equation*}
\end{lemma}

\begin{proof}
By the triangle inequality, we have
\begin{equation*}
  \begin{aligned}
    \left\|s^{\alpha}V\left(\frac{x}{s}\right)-V(x)\right\|_{p} &\leq \left \|s^{\alpha}V\left (\frac{x}{s}\right)-V\left (\frac{x}{s}\right)\right\|_{p}+\left\|V\left(\frac{x}{s}\right)-V(x)\right\|_{p}\\
    &=(s^{\alpha}-1)s^{\frac{N}{p}}\left\|V(x)\right\|_{p}+\left\|V\left(\frac{x}{s}\right)-V(x)\right\|_{p}.
    \end{aligned}
\end{equation*}
Then, since $C^{\infty}_{0}(\R^{N})$ is dense in $L^{p}(\R^{N})$ for every $ p < +\infty$, for every 
$\varepsilon>0$ there exists $H \in C^{\infty}_{0}(\R^{N})$ such that, for $s$ close to $1$,
\begin{equation*}
  \left\|H(x)-V(x)\right\|_{p}< \varepsilon, \ \ \ \ \ \mbox{and } \ \ \ \left\|H\left(\frac{x}{s}\right)-V\left(\frac{x}{s}\right)\right\|_{p}< 2\varepsilon.
\end{equation*}
Since $H \in C^{\infty}_{0}(\R^{N})$, by the H\"{o}lder inequality (with $\frac{1}{p}+\frac{1}{p^{\prime}}=1$) we have
\begin{equation*}
  \begin{aligned}
  \int_{\R^{N}}\left| H(sx)-H(x)\right|^{p}dx &\leq \int_{\R^{N}}\left|\int_{1}^{s}|\nabla H(zx)\cdot x|dz\right|^{p}dx\\
  &\leq |s-1|^{\frac{p}{p^{\prime}}} \int_{1}^{s}\left( \int_{\R^{N}}|\nabla H(zx)\cdot x|^{p}dx \right)dz\\
  &=|s-1|^{\frac{p}{p^{\prime}}} \int_{1}^{s}\left( \int_{\R^{N}}\left|\nabla H(y)\cdot \frac{y}{z}\right|^{p}z^{-N}dy \right)dz\\
  &=|s-1|^{\frac{p}{p^{\prime}}} \int_{1}^{s}z^{-N-p}dz \int_{\R^{N}}|\nabla H(y)\cdot y|^{p}dy\\
  & \leq C\left\|\nabla H(y)\cdot y \right\|_{p}^{p} |s-1|^{\frac{p}{p^{\prime}}+1}<\varepsilon^{p}, 
  \end{aligned}
\end{equation*} 
for $|s-1|$ sufficiently small. Summing up, we have
\begin{equation*}
\begin{aligned}
\left\|V\left(\frac{x}{s}\right)-V(x)\right\|_{p} \leq \left\|V\left(\frac{x}{s}\right)-H\left(\frac{x}{s}\right)\right\|_{p} + \left\|H(x)-V(x)\right\|_{p} + \left\|H\left(\frac{x}{s}\right)-H(x)\right\|_{p} <4 \eps,
\end{aligned}
\end{equation*}
and the lemma follows.
\end{proof}

In order to estimate $c_\mu$, we argue as in \cite[Lemma 3.1]{MR4433054} (see also 
\cite[Proposition 1.12]{MR4476243}), although we need to modify some parts of the strategy. Let us introduce the function $U_{\varepsilon} \in H^{1}(\R^{N})$ defined as
  \begin{equation}\label{BrL}
    U_{\varepsilon}=\frac{\eta [N(N-2)\varepsilon^{2}]^{\frac{N-2}{4}}}{[\varepsilon^{2}+|x|^{2}]^{\frac{N-2}{2}}},
  \end{equation}
where $\eta \in C^{\infty}_{0}(\R^{N})$, $0\leq\eta\leq 1,$ is a smooth cut-off function such that $\eta\equiv1$ in $B_{R}$, $\eta\equiv0$ in $\R^{N}\setminus B_{3R/2}$ (with $B_{2R}$ as in assumption \eqref{eq:final_ass}).  Then, using the estimates provided by Struwe \cite[page 179]{Struwebook} (or those by
Brezis and Nirenberg, \cite[eqs. (1.13), (1.29)]{BrezisNirenberg}), we have
\begin{equation}\label{est:struwe}
    \begin{aligned}\displaystyle
&\|\nabla U_{\varepsilon}\|_{2}^{2}=S^{N/2}+O(\varepsilon^{N-2}), \ \ \ \|U_{\varepsilon}\|_{2^{*}}^{2^{*}}=S^{N/2}+O(\varepsilon^{N}),\\
&\|U_{\varepsilon}\|_{2}^{2}=
\begin{cases}
c \varepsilon^{2}+O(\varepsilon^{3}), & \text{if }N=5,\\
c \varepsilon^{2}|\ln \varepsilon|+O(\varepsilon^{2}), & \text{if }N=4,\\
c \varepsilon+O(\varepsilon^{2}), & \text{if } N=3,
\end{cases}
    \end{aligned}
\end{equation}
where $c$ denotes a strictly positive constant (depending on N). Also, we have
\begin{equation}\label{est:Up}
    \begin{aligned}\displaystyle
\|U_{\varepsilon}\|_{p}^{p}=
\begin{cases}
 c_{1}\varepsilon^{N-\frac{N-2}{2}p}+o(\varepsilon^{N-\frac{N-2}{2}p}), & \text{if }\frac{N}{N-2}<p<2^{*},\\
 c_{1}\varepsilon^{\frac{N}{2}}|\ln \varepsilon|+O(\varepsilon^{\frac{N}{2}}), & \text{if }p=\frac{N}{N-2},\\
c_{1}\varepsilon^{\frac{N-2}{2}p}+o(\varepsilon^{\frac{N-2}{2}p}), & \text{if } 1\leq p<\frac{N}{N-2}.
\end{cases}
    \end{aligned}
\end{equation}
where $c_{1}$ denotes a strictly positive constant (depending on $N,p$). We have the following lemma.
\begin{lemma}\label{upbound_moun}
  Under the assumptions of Theorem \ref{thm:mountain_pass}, we have
  \begin{equation}\label{EST_MP}
    c_{\mu}< m_{\mu}+\frac{1}{N}S^{\frac{N}{2}}\mu^{1-\frac{N}{2}}< \frac{1}{N}S^{\frac{N}{2}}\mu^{1-\frac{N}{2}}.
  \end{equation}
\end{lemma}

\begin{proof}
First, we define $W_{t,\varepsilon}=u_{\mu}+tU_{\varepsilon}$, where $u_{\mu}$ is a ground state solution, according to Theorem \ref{Thm_neg}, $U_{\varepsilon}$ is defined in \eqref{BrL} and $t\geq 0$. For any function $v$, let us define $\tilde v$ as
\[
\tilde{v}(x)=s^{\frac{N-2}{2}}v(sx),\qquad\text{ where $s=\|W_{t,\varepsilon}\|_{2}>0$,}
\]
so that 
\[
\tilde{W}_{t,\varepsilon}=\tilde{u}_{\mu}+t\tilde{U}_{\varepsilon},\qquad\text{ with }\|\tilde{W}_{t,\varepsilon}\|_{2}=1.
\]

Subsequently, we consider the relationship between $c_{\mu}$ and $E_{\mu}(\tilde{W}_{t,\varepsilon})$. Fix $\varepsilon$ and let $t\rightarrow 0$, it is evident that
\begin{equation*}
 \tilde{W}_{t,\varepsilon}\to u_\mu,\qquad\text{and }  E_{\mu}(u_{\mu})<0, \quad\|\nabla u_\mu 
 \|_{2}\leq \bar{R}
\end{equation*}
(recall \eqref{eq:Rbar}). Similarly, when $t\rightarrow +\infty$, we have
\begin{equation*}
  E_{\mu}(\tilde{W}_{t,\varepsilon})\rightarrow -\infty \ \ \mbox{with } \ \|\nabla\tilde{W}_{t,\varepsilon}\|_{2}\rightarrow +\infty.
\end{equation*}
Thus, we can take $t_{1}(\varepsilon)\ll 1$ and $t_{2}(\varepsilon)\gg 1$ such that
\begin{equation}\label{MP_COMP}
  c_{\mu}\leq \max_{t_{1}(\varepsilon)\leq t\leq t_{2}(\varepsilon)}E_{\mu}(\tilde{W}_{t,\varepsilon}).
\end{equation}

First, since  $u_{\mu}\in L^\infty(\R^N)$, by \eqref{est:Up} we have
\begin{equation}\label{est_u1U}
  \int_{\R^{N}}u_{\mu}U_{\varepsilon}^{2^{*}-1}dx = c_{2} \varepsilon^{\frac{N-2}{2}}+ o(\eps^{\frac{N-2}{2}}),
\end{equation}
where $c_{2}$ denotes a strictly positive constant (depending on $N$). By the definition of $s$, we know that
\begin{equation}\label{eq:def_s}
  s^{2}=\|u_{\mu}+tU_{\varepsilon}\|_{2}^{2}= 1+t^{2}\|U_{\varepsilon}\|_{2}^{2}+2t \int_{\R^{N}} u_{\mu}U_{\varepsilon}\,dx,
\end{equation}
which implies that
\begin{equation}\label{est_s}
\begin{aligned}
  s-1=\frac{s^{2}-1}{s+1}=&   \frac{1}{s+1}\left\{ t^{2} \|U_{\varepsilon}\|_{2}^{2}+2t \int_{\R^{N}} u_{\mu}U_{\varepsilon}\,dx \right\}.
\end{aligned}
\end{equation}
Since $3\leq N \leq 5$, we know that $(1+t)^{2^{*}}\geq 1+ t^{2^{*}}+ 2^{*} t+ 2^{*} t^{2^{*}-1}$ 
for every $t \geq 0$. 
By direct computation, we have
\begin{equation*}
  \begin{aligned}
  E_{\mu}(\tilde{W}_{t,\varepsilon})\leq & \frac{1}{2}\int_{\R^{N}}|\nabla \tilde{u}_{\mu}|^{2}dx+ t \int_{\R^{N}}\nabla \tilde{u}_{\mu}\cdot \nabla \tilde{U}_{\varepsilon}dx+\frac{t^{2}}{2}\int_{\R^{N}}|\nabla \tilde{U}_{\varepsilon}|^{2}dx\\
  & +\frac{1}{2}\int_{\R^{N}}V \tilde{u}_{\mu}^{2}dx+t \int_{\R^{N}}V \tilde{u}_{\mu}\cdot  \tilde{U}_{\varepsilon}dx+\frac{t^{2}}{2}\int_{\R^{N}} V \tilde{U}_{\varepsilon}^{2}dx\\
  & -\frac{\mu}{2^{*}}\int_{\R^{N}}\tilde{u}_{\mu}^{2^{*}}dx-\frac{\mu t^{2^{*}}}{2^{*}}\int_{\R^{N}}\tilde{U}_{\varepsilon}^{2^{*}}dx-\mu t \int_{\R^{N}} \tilde{u}_{\mu}^{2^{*}-1} \tilde{U}_{\varepsilon}dx- \mu t^{2^{*}} \int_{\R^{N}} \tilde{u}_{\mu} \tilde{U}_{\varepsilon}^{2^{*}-1}dx,
  \end{aligned}
\end{equation*}
which yields
\begin{equation*}
  \begin{aligned}
  E_{\mu}(\tilde{W}_{t,\varepsilon})\leq& E_{\mu}(\tilde{u}_{\mu}) +\frac{t^{2}}{2}\int_{\R^{N}}|\nabla \tilde{U}_{\varepsilon}|^{2}dx-\frac{\mu t^{2^{*}}}{2^{*}}\int_{\R^{N}}\tilde{U}_{\varepsilon}^{2^{*}}dx\\
    & + t \int_{\R^{N}}\nabla \tilde{u}_{\mu}\cdot \nabla \tilde{U}_{\varepsilon}dx+t \int_{\R^{N}}V \tilde{u}_{\mu}\cdot  \tilde{U}_{\varepsilon}dx-\mu t \int_{\R^{N}} \tilde{u}_{\mu}^{2^{*}-1} \tilde{U}_{\varepsilon}dx\\
  & +\frac{t^{2}}{2}\int_{\R^{N}} V \tilde{U}_{\varepsilon}^{2}dx- \mu t^{2^{*}} \int_{\R^{N}} \tilde{u}_{\mu} \tilde{U}_{\varepsilon}^{2^{*}-1}dx.
  \end{aligned}
\end{equation*}
Since $u_{\mu}(y)=s^{-\frac{N-2}{2}}\tilde{u}_{\mu}(s^{-1}y)$ and $u_{\mu}$ is the ground state 
solution of \eqref{eq:main1}, with Lagrange multiplier 
$\lambda_\mu<0$, we have
\begin{equation*}
  \int_{\R^{N}}\nabla \tilde{u}_{\mu}\nabla \tilde{U}_{\varepsilon}dx+s^{2}\int_{\R^{N}}V(sx)\tilde{u}\tilde{U}_{\varepsilon}dx-\mu \int_{\R^{N}}\tilde{u}_{\mu}^{2^{*}-1}\tilde{U}_{\varepsilon}dx=\lambda_{\mu} s^{2} \int_{\R^{N}}\tilde{u}_{\mu}\tilde{U}_{\varepsilon}dx.
\end{equation*}
From this, we deduce
\begin{equation}\label{eq:before_divide}
  \begin{aligned}
  E_{\mu}(\tilde{W}_{t,\varepsilon})\leq& m_{\mu}+ E_{\mu}(\tilde{u}_{\mu})-E_{\mu}(u_{\mu}) +\frac{t^{2}}{2}\int_{\R^{N}}|\nabla \tilde{U}_{\varepsilon}|^{2}dx-\frac{\mu t^{2^{*}}}{2^{*}}\int_{\R^{N}}\tilde{U}_{\varepsilon}^{2^{*}}dx\\
    & + t \int_{\R^{N}}\left[V(x)-s^{2}V(sx)\right] \tilde{u}_{\mu}\cdot  \tilde{U}_{\varepsilon}dx+t \lambda_{\mu} s^{2} \int_{\R^{N}}\tilde{u}_{\mu}\tilde{U}_{\varepsilon}dx\\
  & +\frac{t^{2}}{2}\int_{\R^{N}} V \tilde{U}_{\varepsilon}^{2}dx-\mu t^{2^{*}} \int_{\R^{N}} \tilde{u}_{\mu} \tilde{U}_{\varepsilon}^{2^{*}-1}dx.
  \end{aligned}
\end{equation}
Now, on the one hand, one can see that there exists $t_{0}>0$ such that
  \begin{equation*}
    c_{\mu}< m_{\mu}+\frac{1}{N}S^{\frac{N}{2}}\mu^{1-\frac{N}{2}},
  \end{equation*}
for $0<t<\frac{1}{t_{0}}$ or $t>t_{0}$ and $\eps$ sufficiently small, just depending on $t_0$. On the other hand, for $ \frac{1}{t_{0}} \leq t \leq t_{0}$, we estimate the right-hand side separately. To begin with, by \eqref{est:struwe}, we have
\begin{equation*}
\begin{aligned}
  \max_{t} \left(\frac{t^{2}}{2}\int_{\R^{N}}|\nabla \tilde{U}_{\varepsilon}|^{2}dx-\frac{\mu t^{2^{*}}}{2^{*}}\int_{\R^{N}}\tilde{U}_{\varepsilon}^{2^{*}}dx \right)=&\max_{t} \left( \frac{t^{2}}{2}\int_{\R^{N}}|\nabla U_{\varepsilon}|^{2}dx-\frac{\mu t^{2^{*}}}{2^{*}}\int_{\R^{N}}U_{\varepsilon}^{2^{*}}dx\right)\\
   =& \frac{1}{N}S^{\frac{N}{2}}\mu^{1-\frac{N}{2}}+O(\varepsilon^{N-2}).
   \end{aligned}
\end{equation*}
Then,  \eqref{eq:before_divide} writes
\begin{equation}\label{eq:after_divide}
  E_{\mu}(\tilde{W}_{t,\varepsilon})\leq m_{\mu}+\frac{1}{N}S^{\frac{N}{2}}\mu^{1-\frac{N}{2}}+O(\varepsilon^{N-2}) + I + II + III,
\end{equation}
where
\begin{equation*}
\begin{split}
I:=&E_{\mu}(\tilde{u}_{\mu})-E_{\mu}(u_{\mu})+t \lambda_{\mu} s^{2} \int_{\R^{N}}\tilde{u}_{\mu}\tilde{U}_{\varepsilon}dx,\\
II:=&t \int_{\R^{N}}\left[V(x)-s^{2}V(sx)\right] \tilde{u}_{\mu}\cdot  \tilde{U}_{\varepsilon}dx,\\
III:=&\frac{t^{2}}{2}\int_{\R^{N}} V \tilde{U}_{\varepsilon}^{2}dx-\mu t^{2^{*}} \int_{\R^{N}} \tilde{u}_{\mu} \tilde{U}_{\varepsilon}^{2^{*}-1}dx.
\end{split}
\end{equation*}
First, we estimate $I$. We have 
\begin{equation*}
\begin{aligned}
 E_{\mu}(\tilde{u}_{\mu})-E_{\mu}(u_{\mu})&=\int_{\R^{N}}V(x)\left[s^{N-2}u_{\mu}^{2}(sx)-u^{2}_{\mu}(x)\right]dx\\
  &=\int_{\R^{N}}V(x)\int_{1}^{s}\left[ \frac{N-2}{2}z^{N-3}u_{\mu}^{2}(zx)+z^{N-2}u_{\mu}(zx)\nabla u_{\mu}(zx) \cdot x \right]dz dx\\
  &= \int_{1}^{s} z^{-3}\left(\int_{\R^{N}}V(x)\left( \frac{N-2}{2}u_{\mu}^{2}(zx)+u_{\mu}(zx)\nabla u_{\mu}(zx) \cdot zx \right) d(zx) \right)dz\\
  &= \int_{1}^{s} \int_{\R^{N}}z^{-3}V\left(\frac{x}{z}\right)\left( \frac{N-2}{2}u_{\mu}^{2}(x)+u_{\mu}(x)\nabla u_{\mu}(x) \cdot x \right) dx\, dz.
\end{aligned}
\end{equation*}
Since $u_{\mu}$ solves \eqref{eq:main1}, we know $u_{\mu}$ satisfies \eqref{Pohozaev_mini}, which implies that
\begin{equation*}
\lambda_{\mu} =\frac{2-N}{2}\int_{\R^{N}}V(x)u_{\mu}^{2}dx-\int_{\R^{N}}V(x)u_{\mu}\nabla u_{\mu} \cdot x\,dx.
\end{equation*}
Thus, we obtain
\[
E_{\mu}(\tilde{u}_{\mu})-E_{\mu}(u_{\mu})= \int_{1}^{s} \left(\int_{\R^{N}}\left[z^{-3}V\left(\frac{x}{z}\right)-V(x)\right]\left( \frac{N-2}{2}u_{\mu}^{2}(x)+u_{\mu}(x)\nabla u_{\mu}(x) \cdot x \right) dx-\lambda_{\mu}\right) dz.
\]
Recalling the definition of $s$ in \eqref{eq:def_s}, and since $z\in[1,s]$, we have that, for 
$\varepsilon$ small, for every $\delta >0$ there exists $\sigma>0$ independent of $z$, such that
\begin{equation*}
  \left|\int_{\R^{N}\setminus B_{\sigma}}[z^{-3}V(\frac{x}{z})-V(x)]\left( \frac{N-2}{2}u_{\mu}^{2}(x)+u_{\mu}(x)\nabla u_{\mu}(x) \cdot x \right) dx\right|< \delta.
\end{equation*}
On the other hand, recall that $V=V_{1}+V_{2}$, where $V_{1}\in L^{\frac{N}{2}}(\R^{N})$ and $V_{2} \in L^{\infty}(\R^{N})$. We  obtain $\tilde{V}=V \cdot \chi _{B_{\sigma}} \in L^{\frac{N}{2}}(\R^{N})$. Also, we have $\tilde{W}=W \cdot \chi _{B_{\sigma}} \in L^{N}(\R^{N})$. Thus, by Lemma \ref{Cont_V}, we have
\begin{equation*}
  \left\|z^{-3}\tilde{V}\left(\frac{x}{z}\right)-\tilde{V}(x)\right\|_{L^{\frac{N}{2}}(\R^{N})}\rightarrow0 \ \ \ \mbox{and} \ \ \ \left\|z^{-2}\tilde{W}\left(\frac{x}{z}\right)-\tilde{W}(x)\right\|_{L^{N}(\R^{N})}\rightarrow0,
\end{equation*}
as $z\rightarrow1$. Then, for $\varepsilon$ small, we have
\begin{equation*}
\begin{aligned}
  &\left|\int_{ B_{\sigma}}\left[z^{-3}V\left(\frac{x}{z}\right)-V(x)\right]\left( \frac{N-2}{2}u_{\mu}^{2}(x)+u_{\mu}(x)\nabla u_{\mu}(x) \cdot x \right) dx\right|\\
  &\leq \frac{N-2}{2} \left\|z^{-3}V\left(\frac{x}{z}\right)-V(x)\right\|_{L^{\frac{N}{2}}(B_{\sigma})}\left\|u_{\mu}\right\|^{2}_{2^{*}}+\left\|z^{-2}W\left(\frac{x}{z}\right)-W(x)\right\|_{L^{N}(B_{\sigma})}\left\|u_{\mu}\right\|_{2^{*}}\left\|\nabla u_{\mu}\right\|_{2}\\
  &\leq 2 \delta.
  \end{aligned}
\end{equation*}
From the above estimates, we have $E_{\mu}(\tilde{u}_{\mu})-E_{\mu}(u_{\mu}) \leq \int_{1}^{s} (3\delta -\lambda) dz$, and since $\delta>0$ is arbitrary, we obtain
\begin{equation*}
  E_{\mu}(\tilde{u}_{\mu})-E_{\mu}(u_{\mu}) \le  -\lambda_{\mu}( s-1)+ o(|s-1|),
\end{equation*}
which implies that
\begin{equation*}
I \leq -\lambda_{\mu}( s-1)+t \lambda_{\mu} s^{2} \int_{\R^{N}}\tilde{u}_{\mu}\tilde{U}_{\varepsilon}dx+ o(|s-1|)=-\lambda_{\mu}( s-1)+t \lambda_{\mu}  \int_{\R^{N}}u_{\mu}U_{\varepsilon}dx+ o(|s-1|).
\end{equation*}
Since $3\leq N \leq 5$, using \eqref{est_s} we obtain
\begin{equation}\label{eq:es_I}
I \leq \frac{-\lambda_{\mu}}{s+1} t^{2}\|U_{\eps}\|^{2}_{2}+o(\varepsilon^{\frac{N-2}{2}})=o(\varepsilon^{\frac{N-2}{2}}).
\end{equation}
Turning to $II$, since $V \in L^{\infty}(B_{R})$, we have that $V \in L^{r}(B_{R})$ for $r> \frac{N}{2}$, where $R$ is defined as in assumption \eqref{eq:final_ass}. Then, by Lemma \ref{Cont_V} and \eqref{est:Up}, we obtain that
\begin{equation}\label{eq:es_II}
  \begin{aligned}
 II\leq t\left|\int_{\R^{N}}\left[V(x)-s^{2}V(sx)\right] \tilde{u}_{\mu}\cdot  \tilde{U}_{\varepsilon}dx \right| &\leq \int_{\R^{N}} \left| s^{-2}V\left(\frac{x}{s}\right)-V(x) \right|u_{\mu}U_{\varepsilon}dx\\
 &\leq C\|u_{\mu}\|_{\infty}\left\|s^{-2}V\left (\frac{x}{s}\right)-V(x)\right\|_{r}\|U_{\varepsilon}\|_{r^{\prime}}\\
 &=o(\varepsilon^{\frac{N-2}{2}}),
  \end{aligned}
\end{equation}
where $1< r^{\prime}< \frac{N}{N-2}$. Finally, we consider $III$. By \eqref{eq:final_ass}, 
\eqref{BrL} and \eqref{est:struwe}, recalling that the cut-off $\eta$ is supported in $B_{3R/2}$, 
we infer that
\begin{equation}\label{eq:est_pot}
  \int_{\R^{N}} V \tilde{U}_{\varepsilon}^{2}dx\leq s^{-2}\|V\|_{L^\infty(B_{2R})} \| U_{\varepsilon}\|^{2}_{2}= o(\varepsilon^{\frac{N-2}{2}}),
\end{equation}
as long as $N=3,4,5$. By \eqref{est_u1U}, we have
\begin{equation}\label{eq:es_III}
 III=-c_{2}\varepsilon^{\frac{N-2}{2}}+o(\varepsilon^{\frac{N-2}{2}}).
\end{equation}
Thus, plugging \eqref{eq:es_I}, \eqref{eq:es_II} and \eqref{eq:es_III} into \eqref{eq:after_divide}, we obtain
\begin{equation*}
\begin{aligned}
   E_{\mu}(\tilde{W}_{t,\varepsilon})\leq& m_{\mu}+\frac{1}{N}S^{\frac{N}{2}}\mu^{1-\frac{N}{2}} 
   -c_{2}\varepsilon^{\frac{N-2}{2}}+o(\varepsilon^{\frac{N-2}{2}}),
   \end{aligned}
\end{equation*}
which, combined with \eqref{MP_COMP}, implies that $\eqref{EST_MP}$ is true.
\end{proof}

Subsequently, we present a theorem from Ghoussoub's book \cite[Thm. 4.1]{MR1251958}, 
which is essential for constructing a localized Palais-Smale sequence.
\begin{lemma}[{\cite[Thm. 4.1]{MR1251958}}]\label{closesequence}
Let $X$ be a Hilbert manifold and let $E_{\mu}\in C^{1}(M,\mathbb{R})$ be a given functional. Let $K\subset X$ be compact and consider a subset
\[\Gamma\subset \{\gamma\subset X:\gamma \ \text{is compact }, \ K\subset \gamma\}\]
which is invariant with respect to deformations leaving $K$ fixed. Assume that
\[\max_{u\in K}E_{\mu}(u)< c_{\mu}:=\inf\limits_{\gamma\in \Gamma}\max\limits_{u\in \gamma}E_{\mu}(u).\]
Let $(\gamma_{n})_n\subset \Gamma$ be a sequence such that
\[
\max_{u\in \gamma_{n}}E_{\mu}(u)\to c_{\mu}
\qquad\text{ as }n\to+\infty.
\]
Then there exists a sequence $(v_{n})_n\subset X$ such that, as $n\to+\infty$,
\begin{enumerate}
\item $E_{\mu}(v_{n})\to c_{\mu}$,
\item $\|\nabla_X E_{\mu}(v_{n})\|\to0$,
\item $dist(v_{n},\gamma_{n})\to 0$.
\end{enumerate}
\end{lemma}

In the following lemma, we construct a bounded Palais-Smale sequence made of ``almost positive'' functions. We adopt the proof strategy from \cite[Proposition 3.4.]{MR4443784}, but since some of the conclusions will be needed later, we provide part of the proof here.

\begin{lemma}\label{boundedPS}
There exists a Palais-Smale sequence $(v_{n})_{n}$ for $E_{\mu}$ constrained on $\mathcal{M}$ at the level $c_{\mu}$, namely
\begin{equation}\label{PSlimit}
  E_{\mu}(v_{n})\rightarrow c_{\mu}, \ \ \ \ \nabla_{\mathcal{M}}E_{\mu}(v_{n})\rightarrow 0, \ \ \ \mbox{as } \ n\rightarrow +\infty,
\end{equation}
such that
\begin{equation}\label{Pohozaev}
  \|\nabla v_{n}\|_{2}^{2}-\mu \|v_{n}\|_{2^{*}}^{2^{*}}+\frac{1}{2}\int_{\mathbb{R}^{N}}V(x)(Nv_{n}^{2}+2v_{n} \nabla v_{n} \cdot x)dx\rightarrow 0, \ \ \ \mbox{as } \ n\rightarrow \infty,
\end{equation}
\begin{equation}\label{non_negative}
  \lim_{n\rightarrow +\infty}\|(v_{n})^{-}\|_{2}=0.
\end{equation}
Moreover, under the same conditions in Theorem \ref{Thm_neg}, the sequence $(v_{n})_{n}$ is bounded and the associated Lagrange multipliers $\lambda_{n}$ are bounded too.
\end{lemma}
\begin{proof}
Let $\xi_{n}\in \Gamma$ be such that
\begin{equation*}
  \max_{t\in[0,1]}E_{\mu}(\xi_{n})\rightarrow c_{\mu}, \ \ \ \mbox{as } \ n\rightarrow +\infty.
\end{equation*}
Since $E_{\mu}$ and $\mathcal{M}$ are even, we can assume $\xi_{n}(t)\geq0$ in $\R^{N}$, for every $t$ and $n$. To construct the Palais-Smale sequence, we apply Lemma \ref{closesequence} to $\tilde{E}_{\mu}$. First, we need to check the assumptions in Lemma \ref{closesequence}. We know
\begin{equation*}
  K=\{(w_{h_{0}},0),(w_{h_{1}},0)\}, \ \ \ \Gamma=\tilde{\Gamma}, \ \ \ X:=\mathcal{M}\times\R, \ \ \ \gamma_{n}=\{(\xi_{n}(t),0):t\in[0,1]\}.
\end{equation*}
 By Lemma \ref{ener_compare}, we know $\tilde{c}_{\mu}=c_{\mu}$, which implies that
\begin{equation*}
  \max_{t\in[0,1]}E_{\mu}(\gamma_{n})\rightarrow \tilde{c}_{\mu}, \ \ \ \mbox{as } \ n\rightarrow +\infty.
\end{equation*} 
Therefore, we can apply Lemma \ref{closesequence} to obtain that there exists a sequence 
$(u_{n},s_{n})_{n} \in \mathcal{M}\times \R$ such that
\begin{equation*}
  \tilde{E}_{\mu}(u_{n},s_{n})\rightarrow c_{\mu}, \ \ \|\nabla_{\mathcal{M}}\tilde{E}_{\mu}(u_{n},s_{n})\|\rightarrow 0 \ \ \|(u_{n},s_{n})-(\xi_{n},0)\|\rightarrow 0, \ \mbox{for some } \ t_{n}\in[0,1].
\end{equation*}
Defining $v_{n}(x)=e^{\frac{Ns}{2}}u_{n}(e^{s}x)$, we readily obtain that $(v_n)_n$ satisfies 
\eqref{PSlimit} and \eqref{non_negative}. At the same time, differentiating $\tilde{E}_{\mu}(u,s)$ with respect to $s$, we obtain \eqref{Pohozaev}.

Subsequently, we consider the boundedness of the Palais-Smale sequence (and of the associated Lagrange multipliers). Setting
\begin{equation*}
  a_{n}:=\|v_{n}\|_{2}^{2}, \ \ \ b_{n}=\|v_{n}\|_{2^{*}}^{2^{*}}, \ \ \ c_{n}:=\int_{\R^{N}}V(x)v_{n}^{2}dx, \ \ \ d_{n}:=\int_{\R^{N}}V(x)v_{n}\nabla v_{n}\cdot x dx,
\end{equation*}
we can deduce that
\begin{equation}\label{pre_ener}
  a_{n}+c_{n}-\frac{2\mu}{2^{*}}b_{n}=2c_{\mu}+o(1),
\end{equation}
\begin{equation}\label{pre_equ}
  a_{n}+c_{n}=\lambda_{n}+\mu b_{n}+o(1)((a_{n}+1)^{1/2}),
\end{equation}
\begin{equation}\label{aft_Pohozaev}
  a_{n}-\mu b_{n}+\frac{N}{2}c_{n}+d_{n}=o(1).
\end{equation}
Combining \eqref{pre_ener} and \eqref{aft_Pohozaev}, we obtain
\begin{equation*}
  \frac{2\mu}{N}b_{n}=2c_{\mu}+\frac{N-2}{2}c_{n}+d_{n}+o(1).
\end{equation*}
Then, by \eqref{pre_ener}, we have
\begin{equation*}
\begin{aligned}
  a_{n}&=Nc_{\mu}+\left(\left(\frac{N-2}{2}\right)^{2}-1\right)c_{n}+\frac{N-2}{2}d_{n}+o(1)\\
  &=Nc_{\mu}+\frac{N(N-4)}{4}c_{n}+\frac{N-2}{2}d_{n}+o(1).
\end{aligned}
\end{equation*}
Set $V^{*}_{i}:=\left[(N-4)V_{i}(x)\right]^{+}$, where $i=1,2$. By Lemma \ref{ineq_VW},  we have
\begin{equation*}
  a_{n}\leq N c_{\mu} +\frac{N}{4}\left(S^{-1}\|V_{1}^{*}\|_{\frac{N}{2}}a_{n}+\|V_{2}^{*}\|_{\infty}\right)+\frac{N-2}{2}\left(S^{-\frac{1}{2}}\|W_{1}\|_{N}a_{n}+\|W_{2}\|_{\infty}a_{n}^{\frac{1}{2}}\right)+o(1).
\end{equation*}
By Lemma \ref{upbound_moun}, we know that $c_{\mu}< \frac{1}{N}S^{\frac{N}{2}}\mu^{1-\frac{N}{2}}$, which implies that
\begin{equation*}
\left(1-\frac{N}{4} S^{-1}\|V_{1}^{*}\|_{\frac{N}{2}}- \frac{N-2}{2} S^{-\frac{1}{2}}\|W_{1}\|_{N} \right)a_{n}-\frac{N-2}{2}\|W_{2}\|_{\infty}a_{n}^{\frac{1}{2}}< S^{\frac{N}{2}}\mu^{1-\frac{N}{2}}+\frac{N}{4}\|V_{2}^{*}\|_{\infty} +o(1).
\end{equation*}
By \eqref{AS_N4GS1}, choosing an appropriate positive constant $C_{1}$, we have
\begin{equation*}
  1-\frac{N}{4} S^{-1}\|V_{1}^{*}\|_{\frac{N}{2}}- \frac{N-2}{2} S^{-\frac{1}{2}}\|W_{1}\|_{N}\geq 1-\frac{N}{4} S^{-1}C_{1}- \frac{N-2}{2} S^{-\frac{1}{2}}C_{1} \geq \delta_{2}>0,
\end{equation*}
where $\delta_{2}$ is a positive constant, independent of both $V_{2}$ and $W_{2}$. Actually, we know 
that $\mu^{1-\frac{N}{2}}=\rho^{-2}$. Therefore, letting $A_{n}=\mu^{-1+\frac{N}{2}}a_{n}=\rho^{2}a_{n}$, 
we know that
\begin{equation*}
  \delta_{2}A_{n}-\frac{N-2}{2}\|W_{2}\|_{\infty}\mu^{\frac{1}{2}(-1+\frac{N}{2})}A_{n}^{\frac{1}{2}}-\frac{N}{4}\|V_{2}^{*}\|_{\infty}\mu^{-1+\frac{N}{2}}< S^{\frac{N}{2}}+o(1).
\end{equation*}
If $A_{n}\leq 1$, then $a_{n}$ is bounded. If instead $A_{n}\geq 1$, by \eqref{AS_N4GS2} we obtain
\begin{equation*}
  \delta_{2}A_{n}-\frac{N-2}{2}\|W_{2}\|_{\infty}\mu^{\frac{1}{2}(-1+\frac{N}{2})}A_{n}^{\frac{1}{2}}-\frac{N}{4}\|V_{2}^{*}\|_{\infty}\mu^{-1+\frac{N}{2}}\geq  \left(\delta_{2} - \frac{N-2}{2} C_{2}- \frac{N}{4} C_{2}\right)A_{n} \geq\frac{\delta_{2}}{2}A_{n},
\end{equation*}
for a suitable choice of $C_2>0$. Resuming,
\begin{equation*}
  a_{n}\leq \mu^{1-\frac{N}{2}}\cdot \max\left\{1,\frac{2}{\delta_{2}} S^{\frac{N}{2}}\right\}.
\end{equation*}
Thus, we can conclude that the sequence $a_{n}$ is bounded. Finally, by \eqref{pre_equ}, the sequence 
$(\lambda_{n})_{n}$ is bounded too.
\end{proof}

From the above lemma, we obtain that both $(u_{n})_{n}$ and the corresponding sequence of multipliers $(\lambda_{n})_{n}$ are bounded. Then, we can deduce that there exist $u^{0}$ and $\lambda^{0}$ such that, 
up to subsequences, 
\begin{equation*}\label{MP_WK}
  u_{n}\rightharpoonup u^{0}\ \ \ \mbox{weakly in } \ H^{1}(\R^{N}), \ \ \ \ \lambda_{n}\rightarrow \lambda^{0} \ \ \ \mbox{in } \ \R.
\end{equation*}

Then, we adopt an argument similar to Lemma \ref{cor:split} to rule out Case iii) in Lemma \ref{cor:split}, namely $\lambda^{0}\geq 0$. More precisely, we have the following lemma.

\begin{lemma}\label{lemma:MPlambdaN}
The weak limit  $u^{0}$ of the Palais-Smale sequence $(u_{n})_{n}$ found in Lemma \ref{boundedPS} is not trivial.
\end{lemma}
\begin{proof}
We assume by contradiction that $u^{0}\equiv 0$ and we argue as in Lemma \ref{cor:split}. 

By Lemma \ref{lem:VW_loss} and Remark \ref{rmk:vanishing_V}, we have that, under the contradiction assumption, $(u_{n})_{n}$ is a Palais-Smale sequence for $E_{\mu,\infty}$, as defined in \eqref{func_without_V}. Thus, $(u_{n})_{n}$ satisfies
\begin{equation}\label{eq:lossV}
  \|\nabla u_{n}\|^{2}_{2}=\lambda_{n}+\mu \|u_{n}\|^{2^{*}}_{2^{*}}+o(1).
\end{equation}
By the Pohozaev identity and Lemma \ref{lem:VW_loss}, we have
\begin{equation*}
\lambda_n = \frac{2-N}{2}\int_{\mathbb{R}^{N}}V(x)u_{n}^{2}dx-\int_{\mathbb{R}^{N}} V(x)v_{n} \nabla u_{n} \cdot x\,dx\rightarrow 0 = \lambda^{0},
\end{equation*}
so that equation \eqref{eq:lossV} yields
\[
E_{\mu}(u_{n}) = \frac{\mu}{N}\|u_{n}\|^{2^{*}}_{2^{*}}+o(1).
\]
Since $E_{\mu}(u_{n}) \to c_\mu > 0$ by \eqref{eq:positive_c_mu}, there exists a suitable constant $d_2>0$, such that
\begin{equation*}\label{crtical_itemMP}
  \|u_{n}\|_{2^{*}}\geq d_{2}>0.
\end{equation*}
By Lemma \ref{Concom} there exist sequences $(\alpha_{n})_n$, $(x_{n})_n$, and a non-trivial $u^{0,1}\in H^1(\R^N)$ such that
\begin{equation*}\label{seq_dilationMP}
  \tilde{u}_{n}=\alpha_{n}^{-\frac{N-2}{2}}u_{n}\left(\frac{\cdot-x_{n}}{\alpha_{n}}\right)\rightharpoonup u^{0,1}
  \qquad\text{weakly in }H^{1}(\mathbb{R}^{N}).
\end{equation*}
Then we distinguish two different cases for $\alpha_{n}$. 

\underline{Case 1}: $\alpha_{n}\rightarrow 0$. Since $\|u_{n}\|^{2}_{2}=1$ and \eqref{norm_dilation}, we know
\begin{equation*}
  \| \tilde{u}_{n}\|^{2}_{2}= \alpha_{n}^{2}\| u_{n}\|^{2}_{2} \rightarrow 0, \ \ \mbox{as} \ \ \alpha_{n}\rightarrow 0,
\end{equation*}
which means that $\tilde{u}_{n}\rightarrow 0$ strongly in $L^{2}(\mathbb{R}^{N})$. Therefore, we can deduce that $u^{0,1}=0$ a.e. in $\mathbb{R}^{N}$, which contradicts the fact that $u^{0,1}$ is not trivial.

\underline{Case 2}: $\alpha_{n}\nrightarrow 0$. We know that, for every $\varphi\in C^{\infty}_{0}(\R^N)$,
\begin{equation*}\label{before_eqMP}
  \int_{\mathbb{R}^{N}}\nabla u_{n}\nabla \varphi\, dx= \mu \int_{\mathbb{R}^{N}} |u_{n}|^{2^{*}-2}u_{n} \varphi\, dx + \lambda_{n} \int_{\mathbb{R}^{N}} u_{n}\varphi\, dx+o\left(\|\varphi\|_{H^{1}}\right), \ \ \ \mbox{as} \ n\rightarrow +\infty.
\end{equation*}
Fix $\psi(x) \in C^{\infty}_{0}(\mathbb{R}^{N})$ and define $\varphi_n(y)=\psi(\alpha_{n} y)$. Then, we have
\begin{equation*}\label{relation_varpsiMP}
  \|\varphi_{n}\|_{2}^{2}= \alpha_{n}^{-N}\|\psi\|^{2}_{2}, \ \ \ \|\nabla \varphi_{n}\|^{2}_{2}=\alpha_{n}^{2-N}\|\nabla\psi\|^{2}_{2}.
\end{equation*}
Following this, we find
\begin{equation*}
\begin{aligned}
\left| \int_{\mathbb{R}^{N}}\nabla \tilde{u}_{n} \nabla \psi\,dx -\mu  \int_{\mathbb{R}^{N}}|\tilde{u}_{n}|^{2^{*}-2}\tilde{u}_{n} \psi\,dx \right|&=\left|\alpha_{n}^{\frac{N}{2}-1} \lambda_{n}\int_{\mathbb{R}^{N}} u_{n}\varphi_n\, dx\right|+\alpha_{n}^{\frac{N}{2}-1}o\left(\|\varphi_n\|_{H^{1}}\right)
\\&\leq  \alpha_{n}^{\frac{N}{2}-1}\lambda_{n}\| u_{n}\|_{2}\|\varphi_n\|_{2}+\alpha_{n}^{\frac{N}{2}-1}o\left(\|\varphi_n\|_{H^{1}}\right)\\
 &\leq C\lambda_{n}\alpha_{n}^{-1}+o\left(\|\nabla \psi \|_{2} + \alpha_n^{-1}
  \|\psi \|_{2}\right) \rightarrow 0,
 \end{aligned}
\end{equation*}
as $\frac{\lambda_{n}}{\alpha_{n}}\rightarrow 0$, by the H\"{o}lder inequality and $\| u_{n}\|_{2}=1$. 
Thus, we obtain that $\tilde{u}_{n}$ satisfies
\begin{equation*}\label{eq_splitMP}
  \int_{\mathbb{R}^{N}}|\nabla \tilde{u}_{n}|^{2}\,dx -\mu  \int_{\mathbb{R}^{N}}|\tilde{u}_{n}|^{2^{*}}\,dx=o(1).
\end{equation*}
We have
\begin{equation*}
  E_{\mu,\infty}(\tilde{u}_{n})=\left(\frac{1}{2}-\frac{1}{2^{*}}\right)\|\nabla \tilde{u}_{n}\|_{2}^{2}+o(1)=\frac{1}{N}\|\nabla \tilde{u}_{n}\|_{2}^{2}+o(1).
\end{equation*}
Then
\begin{equation*}
 E_{\mu,\infty}(u_{n})= E_{\mu,\infty}(\tilde{u}_{n})=\frac{1}{N}\|\nabla \tilde{u}_{n}\|_{2}^{2}+o(1)\geq \frac{1}{N}\|\nabla u^{0,1}\|_{2}^{2}+o(1) \geq \frac{1}{N}S^{\frac{N}{2}}\mu^{1-\frac{N}{2}}+o(1),
\end{equation*}
which, together with 
\begin{equation*}
E_\mu(u_n) \to c_{\mu}< m_{\mu}+\frac{1}{N}S^{\frac{N}{2}}\mu^{1-\frac{N}{2}} < \frac{1}{N}S^{\frac{N}{2}}\mu^{1-\frac{N}{2}} ,
\end{equation*}
gives a contradiction.
\end{proof}

To conclude the proof of the existence of a normalized mountain-pass solution, we have to show that $c_{\mu}$ is achieved. To begin with, by Lemma \ref{lem:max_Princ} and \ref{lemma:MPlambdaN}, we obtain that $\lambda_{n} \to \lambda^{0} <0$. Moreover, we know that the Palais-Smale sequence 
$(u_{n})_{n}$ satisfies $u_{n}\rightharpoonup u^{0}$ weakly in $H^{1}(\R^{N})$, where $u^0$ is a 
nonnegative, nontrivial solution of
\[
\begin{cases}
-\Delta u +V(x)u = \mu |u|^{2^{*}-2}u + \lambda u & \text{in }\mathbb{R}^{N},\smallskip\\
\int_{\mathbb{R}^{N}} u^2\,dx = a^2\le 1.
\end{cases}
\]
Assume by contradiction that $(u_n)_n$ does not converge strongly. By Lemma \ref{cor:split}, we have 
that case ii) holds, which yields that
\begin{equation*}
  c_{\mu} \geq E_{\mu}(u^{0})+\frac{1}{N}S^{N/2}\mu^{1-\frac{N}{2}}.
\end{equation*}
On the other hand, by Lemma \ref{upbound_moun},  we have
\begin{equation*}
  c_{\mu} < m_{\mu}+\frac{1}{N}S^{N/2}\mu^{1-\frac{N}{2}},
\end{equation*}
a contradiction with Corollary \ref{coro:GS}. Therefore, we obtain $u_{n}\rightarrow u^{0}$ in $H^{1}(\R^{N})$. Taking into account the change of variables in \eqref{trans},  this completes the proof of the theorem.

\begin{remark}\label{rmk:endofpaper}
With a closer look at the proof of Lemma \ref{upbound_moun}, one can see that assumption 
\eqref{eq:final_ass} on the potential $V$ can be slightly weakened. Indeed, instead of requiring 
$V \in L^{\infty}(B_{2R})$, for some $R>0$, it is enough to take 
$V \in L^{r}(B_{2R})$ where
\begin{equation}\label{special_Assum}
\begin{cases}
  r>2, & \mbox{if } N=3 \\
  r>4, & \mbox{if } N=4 \\
  r>10, & \mbox{if } N=5.
\end{cases}
\end{equation}
If we directly calculate \eqref{eq:est_pot} through the H\"{o}lder inequality and \eqref{est:Up}, we can obtain the results in \eqref{eq:est_pot}, under the assumption \eqref{special_Assum}.
\end{remark}
We conclude with the proof of our application to ergodic Mean Field Games with quadratic cost.
\begin{proof}[Proof of Theorem \ref{thm:MFG}]
Let us substitute the Hopf-Cole transform 
\[
v^2(x)=m(x) = A e^{-u(x)},
\]
where $A$ is a normalization constant, into system \eqref{e:sys}. We obtain that the second equation 
in \eqref{e:sys} is automatically satisfied, while the other conditions rewrite as 
\[
\begin{cases}
\displaystyle -\Delta v +V(x) v= \frac{\alpha}{2} v^{2^*-1} + \frac{\lambda}{2} v \qquad & \text{in} \ \R^N \smallskip\\
\displaystyle \int_{\R^N}v^2=1,\qquad v>0,
\end{cases}
\]
that is, $(v,\frac{\lambda}{2})$ has to solve \eqref{eq:main1}, with $\mu=\frac{\alpha}{2}$. As a
consequence, the required results follow by Theorems \ref{Thm_neg}, \ref{thm:mountain_pass}, the change 
of variables \eqref{trans}, Remark \ref{EX_po} (with Lemma \ref{lem:VinLr}) and Remark 
\ref{rmk:regmax}. In particular, the constant $\alpha^*>0$ depends on the norms of the decompositions 
of $V$, $W$, according to Lemma \ref{lem:VinLr}.
\end{proof}

\textbf{Acknowledgments.} Work partially supported by: PRIN-20227HX33Z ``Pattern formation in nonlinear 
phenomena'' - funded  by the European Union-Next Generation EU, Miss. 4-Comp. 1-CUP D53D23005690006; 
the Portuguese government through FCT/Portugal under the project PTDC/MAT-PUR/1788/2020; 
the MUR grant Dipartimento di Eccellenza 2023-2027; 
the INdAM-GNAMPA group.\bigskip

\textbf{Data Availability.} Data sharing not applicable to this article as no datasets were generated or analyzed during the current study.

\bigskip

\textbf{Disclosure statement.} The authors report there are no competing interests to declare.

\bibliography{normalized}{}
\bibliographystyle{abbrv}
\medskip
\small

\begin{flushright}
{\tt gianmaria.verzini@polimi.it}\\
{\tt junwei.yu@polimi.it}\\
Dipartimento di Matematica, Politecnico di Milano\\
piazza Leonardo da Vinci 32, 20133 Milano, Italy.
\end{flushright}

\end{document}